\documentclass[reqno]{amsart}
\usepackage{graphicx} 
\numberwithin{equation}{section}
\usepackage{amsmath}
\usepackage{amssymb}
\usepackage{amsthm}
\usepackage{color}
\usepackage{enumerate}
\theoremstyle{definition}
\newtheorem{thm}{Theorem}[section]
\newtheorem{prop}[thm]{Proposition}

\newtheorem{rem}[thm]{Remark}
\newtheorem{lemma}[thm]{Lemma}

\usepackage{mathrsfs}

\newcommand{\rational}{R}
\newcommand{\expansive}{\upsilon}
\newcommand{\dimension}{\delta}
\newcommand{\potential}{\phi}
\newcommand{\pinf}{\underline{\alpha}_{\Omega}}
\newcommand{\psup}{\overline{\alpha}_{\Omega}}
\newcommand{\prange}{F}
\newcommand{\indupote}{\tilde \potential}
\newcommand{\smallnumber}{\tilde \beta(\rational)}
\newcommand{\edge}{\Sigma}
\newcommand{\MS}{\edge_A}
\newcommand{\coding}{\pi}
\newcommand{\shift}{\sigma}
\newcommand{\comp}{\text{Comp}}
\newcommand{\Koebe}{\mathfrak{d}}
\newcommand{\induedge}{\tilde \edge}
\newcommand{\indushift}{\tilde \shift}
\newcommand{\CMS}{\tilde \edge_B}
\newcommand{\return}{\rho}
\newcommand{\indudomain}{\mathcal{D}}

\newcommand{\Proj}{\text{Proj}}
\newcommand{\indumeasure}{\tilde \mu}
\newcommand{\indumap}{\tilde R}
\newcommand{\exponent}{\beta}
\newcommand{\indupressure}{\tilde p}
\newcommand{\pressure}{p}
\newcommand{\codingpressure}{p^*}
\newcommand{\conv}{\text{Conv}(\{\delta_{\tau_\omega}\}_{\omega\in\Omega, \tau_\omega\in \coding^{-1}(\omega)})}
\newcommand{\LB}{\text{LB}(q)}
\newcommand{\inddiferent}{\mathcal{I}}
\newcommand{\parabolic}{\gamma}
\newcommand{\rate}{p}
\newcommand{\codmeas}{\mu^*}
\newcommand{\minrate}{\xi}
\newcommand{\growth}{\minrate(\omega)}
\newcommand{\palpha}{p_\alpha}
\newcommand{\minomega}{\underline{\omega}}
\newcommand{\maxomega}{\overline{\omega}}
\newcommand{\vspectrum}{b_{\text{v}}}
\newcommand{\CMSn}{\tilde \edge_{B,n}}

\title[Thermodynamic formalism and multifractal analysis]{Thermodynamic formalism and multifractal analysis of Birkhoff averages for parabolic rational maps}
\author{Yuya Arima}
\date{\today}

\address{Graduate School of Mathematics, Nagoya University,
Furocho, Chikusaku, Nagoya, 464-8602, JAPAN} 
\email{yuya.arima.c0@math.nagoya-u.ac.jp}


\begin{document}

\begin{abstract}
In this paper, we study the multifractal analysis of Birkhoff averages for parabolic rational maps. We establish a conditional variational principle and prove the real analyticity and strict monotonicity of the Birkhoff spectrum, as well as the existence and uniqueness of the measure attaining the supremum in the conditional variational principle, on a certain region.
To this end, we prove the existence and uniqueness of an expanding equilibrium measure and the real analyticity of the pressure function on a suitable domain. For parabolic systems, our approach using thermodynamic formalism provides a unified framework for establishing the conditional variational principle and investigating finer properties of the Birkhoff spectrum, including its real analyticity, strict monotonicity, and the existence and uniqueness of a measure attaining the supremum on a certain region.  
\end{abstract}

\maketitle

\section{Introduction}
Let $f$ be a continuous dynamical system on a compact metric space $X$ and let $\phi$ be a continuous potential on $X$. The Birkhoff average of $\phi$ at $x\in X$ is defined by
\[
\lim_{n\to\infty}\frac{1}{n}\sum_{k=0}^{n-1}\phi(f^k(x))
\]
whenever the limit exists. 
Birkhoff averages provide a way to characterize the dynamical system $f$.   
For a given $f$-invariant ergodic Borel probability measure $\mu$ on $X$,
Birkhoff's ergodic theorem yields that  for $\mu$-almost every $x\in X$ the Birkhoff average of $\phi$ at $x$ converges to the space average $\int \phi d\mu$.
Thus, for $\alpha\neq \int \phi d\mu$ the set $B(\alpha)$ of points that the Birkhoff average of $\phi$ converges to $\alpha$ is negligible with respect to $\mu$.
However, there is still a possibility that $B(\alpha)$ might be a large set from another point of view. This raises the following natural questions: What are the exceptional Birkhoff averages? How large is the set $B(\alpha)$? Multifractal analysis enables us to answer these questions by quantifying the size of $B(\alpha)$. We refer the reader to the books by Pesin \cite{Pesinbook} and Barreira \cite{Barreirabook} for an introduction to multifractal analysis. 

In this paper, we study the multifractal analysis of Birkhoff averages for parabolic rational maps. Our approach using thermodynamic formalism enables us not only to establish the conditional variational principle but also to investigate finer properties of the Birkhoff spectrum on a certain region, including its real analyticity and strict monotonicity, as well as the existence and uniqueness of a measure attaining the supremum in the conditional variational principle. It also allows us to overcome the obstruction to constructing analytic inverse branches around rationally indifferent periodic points caused by the accumulation of critical orbits at rationally indifferent periodic points.

Let $\overline{\mathbb{C}}$ be the Riemann sphere equipped with the spherical metric $d_{\overline{\mathbb{C}}}$ and let $\rational:\overline{\mathbb{C}}\rightarrow  \overline{\mathbb{C}}$ be a rational map. 
We denote by $J$ the Julia set of $\rational$ and by $|\rational'(x)|_s$ the norm of the spherical derivative of $\rational$ at $x\in \overline{\mathbb{C}}$. 
In this paper, we always assume that $\text{deg}(\rational)\geq 2$, where $\text{deg}(\rational)$ denotes the degree of $\rational$.

A rational map $\rational$ is said to be expansive if there exists an expansive constant $\expansive>0$ such that for each $x,y\in J$ with $x\neq y$ we have 
$
\sup_{n\geq 0}
d_{\overline{\mathbb{C}}}(\rational^n(x),\rational^n(y))>\expansive.
$
A rational map $\rational$ is said to be parabolic if $\rational$ is expansive but not expanding (see \cite[Definition 25.3.3]{UrbanskiRoyMunday} for the definition of expanding rational maps).
    Let $\rational$ be a parabolic rational map
and let $\Omega(\rational)$ be the set of rationally indifferent periodic points. Since $\rational$ is parabolic, \cite[Corollary 24.2.23 and Theorem 26.2.7]{UrbanskiRoyMunday} implies that $0<\#\Omega(\rational)<\infty$.

Let $\rational$ be a rational map and let $\potential:J\rightarrow \mathbb{R}$ be a continuous function. 
We denote by $\dim_H(\cdot)$ the Hausdorff dimension with respect to the spherical metric $d_{\overline{\mathbb{C}}}$.
We define the Birkhoff spectrum $b:\mathbb{R}\rightarrow [0,\dim_H(J)]$ by
\[
b(\alpha):=\dim_H(B(\alpha)), \text{ where }
B(\alpha):=\left\{x\in J:\lim_{n\to\infty}\frac{1}{n}\sum_{k=0}^{n-1}\potential(R^k(x))=\alpha\right\}.
\]
Let $M(\rational)$ be the set of $\rational$-invariant Borel probability measures on $J$.
We set 
\[
\alpha_{\min}:= \min\left\{\int \phi d\mu:\mu\in M(\rational)\right\}
\text{ and }
\alpha_{\max}:= \max\left\{\int \phi d\mu:\mu\in M(\rational)\right\}.
\]
Then we have 
$B(\alpha)\neq \emptyset$ if and only if $\alpha\in [\alpha_{\min},\alpha_{\max}]$ (see Proposition \ref{prop non empty B and variational B}).

Let $\dimension$ denote the Hausdorff dimension of $J$. 
By \cite[Theorem 8.7]{AaronsonDenkerUrbanski}, there exists a unique nonatomic $\dimension$-conformal measure $m_{\dimension}$ for $\rational$ (see \cite[Definition 29.2.10]{UrbanskiRoyMunday} for the definition of $\dimension$-conformal measure). Furthermore, by \cite[Section 9]{AaronsonDenkerUrbanski}, $\rational$ admits a unique (up to constant) ergodic invariant Borel measure $\mu_{\dimension}$ on $J$ which is equivalent to $m_{\dimension}$. If $\mu_{\dimension}$ is finite, we normalize it so that $\mu_{\dimension}(J)=1$. See \cite[Theorem 9.8]{AaronsonDenkerUrbanski} for necessary and sufficient conditions for $\mu_{\dimension}$ to be finite.

We define 
\begin{align*}
&\pinf:=\min\left\{\frac{1}{\ell}\sum_{k=0}^{\ell-1}\phi(\rational^k(\omega)):\omega\in \Omega(\rational), \ \rational^\ell(\omega)=\omega\right\} \text{ and }
\\&
\psup:=\max\left\{\frac{1}{\ell}\sum_{k=0}^{\ell-1}\phi(\rational^k(\omega)):\omega\in \Omega(\rational), \ \rational^\ell(\omega)=\omega\right\}.
\end{align*}
We set
\begin{align}\label{eq def flat}
\prange:=
\left\{
\begin{array}{cc}
[\pinf,\psup] 
&\text{ if $\mu_\dimension$ is infinite}\\
\left[\min\{\int \phi d\mu_\dimension,\pinf\},
\max\{\int \phi d\mu_\dimension,\psup\}\right] 
& \text{ if $\mu_\dimension$ is finite}
\end{array}
\right ..
\end{align}

We denote by $\mathcal{H}$ the set of continuous functions on $J$ for which the induced potential defined by \eqref{eq def indupote} is locally H\"older (see Section \ref{sec induced systems} for the definition of locally H\"older). The locally H\"older continuity of a potential allows us to use an inducing method. We would like to stress that the geometric potential $\log |R'|_J|_s$ belongs to $\mathcal{H}$ (see Lemma \ref{lemma geometric Holder}). Moreover, if $\phi:J\rightarrow \mathbb{R}$ is H\"older continuous and Lipschitz continuous on a small neighborhood of $\Omega$ then $\phi\in\mathcal{H}$. See Lemma \ref{lemma sufficient condition Holder} for more details on sufficient conditions under which a potential belongs to $\mathcal{H}$.

For $\mu\in M(\rational)$ we denote by $h(\mu)$ the measure-theoretic entropy (see \cite{walters2000introduction}) and define $\lambda(\mu):=\int \log |R'|_sd\mu$. For $\mu\in M(\rational)$ we define the Hausdorff dimension of $\mu$ by 
$
\dim_H(\mu):=\inf\left\{\dim_H(Z):\mu(Z)=1,\ Z\text{ is Borel measurable}\right\}.
$
By \cite[Theorem 28.4.1]{UrbanskiRoyMunday}, if $\mu\in M(\rational)$ is ergodic and satisfies $\lambda(\mu)>0$ then we have 
\begin{align}\label{eq volume lemma}
    \dim_H(\mu)=\frac{h(\mu)}{\lambda(\mu)}.
\end{align}
We are now in a position to state our main theorem. 
\begin{thm}\label{thm main}
    Let $\rational$ be a parabolic rational map with $\text{deg}(\rational)\geq 2$ and let $\phi\in \mathcal{H}$. Then we have the following:
    \begin{itemize}
        \item[(M1)] For any $\alpha\in (\alpha_{\min},\alpha_{\max})\setminus \prange$ there exists a unique ergodic Borel probability measure $\mu_\alpha\in M(\rational)$ such that $\lambda(\mu_\alpha)>0$, $\int \phi d\mu_\alpha=\alpha$ and
        \begin{align}\label{eq def vspectrum}
        \dim_H(\mu_\alpha)=\sup\left\{\frac{h(\mu)}{\lambda(\mu)}:\mu\in M(\rational),\ \lambda(\mu)>0,\ \int\phi d\mu=\alpha \right\}.
        \end{align}
        Moreover, for any $\alpha\in (\alpha_{\min},\alpha_{\max})\setminus \prange$ we have $b(\alpha)=\dim_H(\mu_\alpha)$. 
\item[(M2)] The function $b$ is real-analytic on $(\alpha_{\min},\alpha_{\max})\setminus \prange$.
\item[(M3)] For all $\alpha\in (\alpha_{\min} , \alpha_{\max} ) \setminus \prange$ we have $b(\alpha)<\dimension$. Moreover, $b$ is strictly increasing on $(\alpha_{\min},\min \prange)$ and strictly decreasing on $(\max \prange,\alpha_{\max})$.
    \end{itemize}
\end{thm}

\begin{thm}\label{thm parabolic ragion}
Let $\rational$ be a parabolic rational map with $\text{deg}(\rational)\geq 2$ and let $\phi$ be a continuous function.
Then for any $\alpha\in \prange$ we have $b(\alpha)=\dimension$.
\end{thm}

To prove Theorem \ref{thm main}, we establish the following theorem, which is of independent interest (Theorems \ref{thm equilibrium state rational} and \ref{thm regularity of non induced pressure}).
Theorem \ref{thm thermodynamics intro} generalizes the results of thermodynamic formalism obtained by Haydn and Isola \cite{HaydnIsola}.
For a continuous function $\psi$ on $J$ we define the topological pressure of $\psi$ by 
\[
P(\psi):=\sup\left\{h(\mu)+\int\psi d\mu: \mu\in M(\rational)\right\}.
\]
For $\phi\in \mathcal{H}$ and $(b,q)\in \mathbb{R}^2$ we define $\pressure(b,q):=P(-q\phi-b\log |\rational'|_J|_s)$. We set 
\[
\mathcal{N}:=\left\{(b,q)\in \mathbb{R}^2: \pressure(b,q)>\max\left\{\frac{-q}{\ell}\sum_{k=0}^{\ell-1}\phi(\rational^k(\omega)):\omega\in \Omega(\rational),\ \rational^\ell(\omega)=\omega\right\}\right\}.
\]
\begin{thm}\label{thm thermodynamics intro}
    Let $\rational$ be a parabolic rational map with $\text{deg}(\rational)\geq 2$ and let $\phi\in \mathcal{H}$. For all 
    $(b,q)\in\mathcal{N}$
    there exists a unique ergodic equilibrium measure $\mu_{b,q}\in M(\rational)$ for $-q\phi-b\log |\rational'|_J|_s$ with $\lambda(\mu_{b,q})>0$. Moreover, the function $(b,q)\mapsto \pressure(b,q)$ is real analytic on $\mathcal{N}$ and we have 
    \[
        \frac{\partial}{\partial b}\pressure(b,q)=-\lambda(\mu_{b,q}) \text{ and } \frac{\partial}{\partial q}\pressure(b,q)=-\int \phi d\mu_{b,q} .
    \]
\end{thm}
 Under more general assumptions on the potential, Huneycutt and Thompson \cite{huneycutt2026specification} established a uniqueness of the expanding equilibrium measure. 
They used a specification and orbit-decomposition approach, which differs from the inducing method we use.

For uniformly expanding systems, the multifractal analysis was well understood. See the books by Pesin \cite{Pesinbook} and Barreira \cite{Barreirabook}. However, for parabolic systems, many of the arguments used for uniformly expanding systems break down, leading to different phenomena, such as phase transitions.

We want to compare our results with those of Johansson et al. \cite{JJTOPnonuniformly} and Jaerisch and Takahasi \cite{MixedJT}, who established a conditional variational principle for the Birkhoff spectrum of one-dimensional parabolic systems.
In our setting, the orbits of critical points accumulate on rationally indifferent periodic points (see \eqref{eq accumulation}). This implies that, for a fixed arbitrarily small neighborhood $U$ of a rationally indifferent periodic point, we cannot, in general, define an analytic inverse branch of $\rational^n$ on $U$ for sufficiently large $n\in\mathbb{N}$.
On the other hand, for the one-dimensional parabolic systems studied in \cite{JJTOPnonuniformly, MixedJT}, one can define an inverse branch of the $n$-th iterate around every parabolic periodic point for all $n\in\mathbb{N}$, and this fact plays an important role in their arguments.
Thus, it is not straightforward to adapt their arguments to our setting. Moreover, their approach is based on constructing a sequence of measures. It is difficult to obtain finer properties of the Birkhoff spectrum, such as those established in Theorem \ref{thm main} using their approach. In contrast, for $\alpha\in (\alpha_{\min} , \alpha_{\max} ) \setminus \prange$ we first identify the unique ergodic measure attaining the supremum in \eqref{eq def vspectrum} by adapting the arguments from \cite{arima}. 
This allows us to establish the conditional variational principle and, at the same time, obtain finer properties of the Birkhoff spectrum. 
Moreover, the proof of the conditional variational principle itself is shorter than the existing proofs given in \cite{JJTOPnonuniformly, MixedJT}.

Unfortunately, it is difficult to analyze $b(\alpha)$ for $\alpha\in\prange$ using an inducing method. Indeed, for $\alpha\in\prange$, there is in general no $\rational$-invariant measure $\mu$ such that $\lambda(\mu)>0$ and $\int\phi d\mu=\alpha$. 
However, when $\alpha\in\prange$, we can argue as follows: 
Using measures concentrated on rationally indifferent periodic orbits or $\mu_\dimension$ when it is finite and a certain hyperbolic approximation, for each $\alpha\in\prange$ we can construct  a sequence of ergodic measures $\{\mu_n\}_{n\in\mathbb{N}}\subset M(\rational)$ such that $\lim_{n\to\infty}\dim_H(\mu_n)=\dimension$ and
$
\int\phi d\mu_n\to\alpha
$
sufficiently fast as $n\to\infty$. 
Moreover, for each $n\in\mathbb{N}$ we can approximate $\mu_n$ with a collection of hyperbolic cylinders associated with the Markov partition constructed in Section \ref{sec parabolic rational maps} (see Lemmas \ref{lemma cylinder approximation} and \ref{lemma hyperbolic ergodic approximation}).
This result enables us to overcome the obstruction to constructing inverse branches described above and to prove Theorem \ref{thm parabolic ragion}.

For parabolic rational maps, William \cite{William} studied the local dimension spectrum, while Gelfert, Przytycki, and Rams \cite{GelfertPrzytyckiRams} studied the Lyapunov spectrum. However,  while the Lyapunov spectrum and the local dimension spectrum can be obtained as Legendre transforms of suitable thermodynamic quantities, such a description is not expected for the Birkhoff spectrum in general. Consequently, the arguments used in \cite{GelfertPrzytyckiRams, William} cannot, in general, be directly applied to the Birkhoff spectrum. In particular, our proofs differ from theirs.

The structure of the paper is as follows. 
In Section \ref{sec parabolic rational maps}, we construct a Markov partition and an induced system. 
In Section \ref{sec Thermodynamic formalism for the countable Markov shift}, we introduce the properties of the countable Markov shift induced by a parabolic rational map. 
In Section \ref{sec Thermodynamic formalism for parabolic rational maps}, we develop thermodynamic formalism for parabolic rational maps.
In Section \ref{sec Variational Birkhoff spectrum and Thermodynamic formalism}, we study the variational Birkhoff spectrum.
In Section \ref{sec Conditional variational principle}, we establish the conditional variational principle and complete the proof of Theorem \ref{thm main}.
In Section \ref{sec proof Theorem parabolic region}, we show Theorem \ref{thm parabolic ragion}.
In Section \ref{sec Verification of the Defining Condition}, we give a sufficient condition under which a continuous potential $\phi:J\rightarrow\mathbb{R}$ belongs to $\mathcal{H}$.

\textbf{Notations.} Throughout we shall use the following notation:
For a index set $\mathcal{Q}$ and $\{a_{q}\}_{q\in\mathcal{Q}},\{b_{q}\}_{q\in\mathcal{Q}}\subset[0,\infty]$ we write $a_q\ll b_q$ if there exists a constant $C\geq 1$ such that for all $q\in \mathcal{Q}$ we have $a_q\leq Cb_q$. If we have $a_q\ll b_q$ and $b_q\ll a_q$ then we write $a_q\asymp b_q$.
For a probability space $(X,\mathcal{B})$, a probability measure $\mu$ on $(X,\mathcal{B})$ and a measurable function $\psi:X\rightarrow \mathbb{R}$ we set
$\mu(\psi):=\int \psi d\mu$.
Moreover, for a dynamical system $(X,T)$, a function $\psi:X\rightarrow \mathbb
{R}$ and $n\in\mathbb{N}$ we set 
$S_n(\psi):=\sum_{k=0}^{n-1} \psi\circ T^k$.

\section{Parabolic rational maps}\label{sec parabolic rational maps}
In the following we always assume that $\rational$ is a parabolic rational map with $\deg(\rational)\geq 2$. 
Then, we have $J\neq \overline{\mathbb{C}}$. 
Thus, by conjugating with a Möbius transformation, we may assume without loss of generality that $\infty\notin J$.  
This assumption allows us to work with the Euclidean metric on $\mathbb{C}$ rather than $d_{\overline{\mathbb{C}}}$. For any $z\in J$ we have 
\[
|R'(z)|_s=\frac{1+|z|^2}{1+|\rational (z)|^2}|\rational'(z)|,
\]
where $|\rational'(z)|$ denotes the norm of the Euclidean derivative of $R$ at $z$ and $|z|$ denotes the the Euclidean norm of $z$. Therefore, $\log |R'|_s$ is cohomologous to $\log|R'|$, that is,  
$\log |R'(z)|_s-\log|R'(z)|=\log (1+|z|^2)-\log (1+|R(z)|^2).$ 
Consequently, for any $\mu\in M(R)$ we have $\lambda(\mu)=\mu( \log |\rational'|_s)=\mu(\log |\rational'|)$. 
In what follows, we use the Euclidean derivative and Euclidean metric rather than the spherical derivative and the spherical metric.

Let $\Omega(\rational)$ be the set of rationally indifferent periodic points.  The rationally indifferent periodic point $\omega$ is said to be simple if $\omega$ is a fixed point of $\rational$ and satisfies $\rational'(\omega)=1$. We denote by $\Omega_0(\rational)$ the set of simple rationally indifferent periodic points of $\rational$. Let $j\in \mathbb{N}$. 
Recall that the Julia set of $\rational$ coincides with the Julia set of $\rational^j$ (see, for example, \cite[Theorem 24.1.5]{UrbanskiRoyMunday}). Let $\potential$ be a continuous function $\potential$ on $J$. For any $z\in J$ we have 
\begin{align}\label{eq higher iteration 1}
\lim_{n\to\infty}\frac{1}{n}\sum_{k=0}^{n-1}\phi_j\circ(\rational^j)^k(z)=j\lim_{n\to\infty}\frac{1}{n}\sum_{k=0}^{n-1}\phi\circ\rational^k(z)
\end{align}
if at least one of the two limits exists. 
Here, $\phi_j$ is defined by 
$
\phi_j:=\sum_{r=0}^{j-1}\phi\circ\rational^r.
$
For $\mu\in M(R)$ we have $\mu\in M(\rational^j)$ and 
$
h(\mu,\rational^j)=j
h(\mu,\rational)
\text{ and } 
\mu( \phi_j )=j\mu( \phi )
$
(see \cite[Theorem 4.13]{walters2000introduction}). Here, $h(\mu,\rational^j)$ denotes the measure-theoretic entropy of $\mu$ with respect to $\rational^j$.
Conversely, for $\overline{\mu}\in M(\rational^j)$ we have 
$
h(\overline{\mu},R^j)=jh(\mu,\rational) 
\text{ and }
\overline{\mu}( \phi_j ) =j\mu( \phi ), 
\text{ where }\mu=\frac{1}{j}\sum_{r=0}^{j-1}\overline{\mu}\circ\rational^{-r}.
$
Thus, after replacing $\rational$ by $\rational^q$ for some $q$ such that $\Omega(\rational^q)=\Omega_0(\rational^q)$, we may assume without loss of generality that
\begin{align}\label{eq assumption simple}
    \Omega(\rational)=\Omega_0(\rational).
\end{align}
Throughout what follows, we assume that \eqref{eq assumption simple} holds.

\subsection{Dynamics of parabolic rational maps near rationally indifferent fixed points}\label{sec Dynamics of parabolic rational maps near rationally indifferent fixed points}
In this section, we briefly describe results on the dynamics of $\rational$ near $\omega\in \Omega(\rational)$ and Fatou's flower theorem. For details we refer the reader to \cite[Section 24.2]{UrbanskiRoyMunday}.
We denote by $\text{Crit}(\rational)$ the set of critical points of $\rational$. For $z\in \mathbb{C}$ and $r>0$ we denote by $B(z,r)$ the Euclidean open ball centered at $z$ with radius $r$.
Let $\smallnumber>0$ be a small number satisfying the following condition: 
\begin{itemize}
    \item[(I1)] For all $\omega\in \Omega(\rational)$ there exists a unique connected component $U(\omega,\smallnumber)$ of $\rational^{-1}(B(\omega,\smallnumber))$ containing $\omega$ such that $\text{Crit}(\rational)\cap U(\omega,\smallnumber) =\emptyset$. In particular, the map $\rational:U(\omega,\smallnumber)\to B(\omega,\smallnumber)$ is biholomorphic. We denote its inverse by $\rational_\omega^{-1}$.
    \item[(I2)]  For all $\omega\in \Omega(\rational)$ and $z\in B(\omega,\smallnumber)$ we have 
\begin{align}
    \rational_{\omega}^{-1}(z)=z-a_\omega (z-\omega)^{\rate_\omega+1}+O((z-\omega)^{\rate_\omega+2}),
\end{align} 
for some $\rate_\omega\in \mathbb{N}$ and $a_\omega\in \mathbb{C}\setminus \{0\}$.
\end{itemize}

Let $\omega\in \Omega(\rational)$. 
For $k=0,\cdots,\rate_\omega-1$ we define 
\begin{align*}
    L_{\omega,k}:=\left\{\omega+te^{i\frac{2k\pi-\arg(a_\omega)}{\rate_\omega}}:t>0\right\}.
\end{align*}
 These sets form the  attracting directions of $\rational_{\omega}^{-1}$, that is,  
$
\bigcup_{k=0}^{\rate_\omega-1}L_{\omega,k}=
\{z\in \mathbb{C}: -a_{\omega}(z-\omega)^{\rate_\omega}<0 \}.
$
We fix an angle $\mathfrak{a}\in [0,2\pi)$ satisfying $\mathfrak{a}\notin \bigcup_{\omega\in \Omega(\rational)}\arg(a_\omega^{-1})$.  
In this paper, for $p\in \mathbb{R}$ we denote by $\sqrt[p]{\cdot}$ the holomorphic branch of the $p$-th radical function defined on $\mathbb{C}\setminus (\{z\in \mathbb{C}: \arg(z)=\mathfrak{a}\}\cup\{0\})$ satisfying $\sqrt[p]{1}=1$. Also, we denote by $(\cdot)^{1/p}$ the holomorphic branch of the $p$th-root function defined on $\mathbb{C}\setminus (-\infty,0]$ satisfying $1^{1/p}=1$.
For $\omega\in \Omega(\rational)$ and $k=0,\cdots,\rate_\omega-1$ we define the maps $\rho_{\omega,k}:\mathbb{C}\rightarrow \mathbb{C}$ and $H:\mathbb{C}\setminus (-\infty,0]\rightarrow \mathbb{C}$ by 
\begin{align*}
    \rho_{\omega,k}(z)=\omega+e^{i2k\pi/\rate_\omega}\sqrt[p_{\omega}]{(a_\omega \rate_\omega)^{-1}} z
    \text{ and } H(z):=z^{-1/\rate_\omega}.
\end{align*}
Note that $\rho_{\omega,k}^{-1}(L_{\omega,k})=(0,\infty)$. 
%
We fix $\theta\in (0,\pi).$
For every $\omega\in \Omega(\rational)$, $k\in \{0,\cdots, p_{\omega}-1\}$ and $x> 0$ let 
\begin{align*}
    S^k_\omega(x,\theta):=\rho_{\omega,k}\circ H(S(x,\theta)),
\end{align*}
where $S(x,\theta):=\{z\in \mathbb{C}\setminus \{x\}: -\theta<\arg(z-x)<\theta\}$.
Then, by \cite[Lemma 24.2.3]{UrbanskiRoyMunday}, there exists $x(\theta)\in (0,\infty)$ such that for all $\omega\in \Omega(\rational)$, $0\leq k\leq \rate_\omega-1$ and $x\geq x(\theta)$ 
\begin{align}\label{eq action rational inverse branch}
    R_{\omega}^{-1}(S^k_\omega(x,\theta))\subset S^k_\omega(x+1/2,\theta).
\end{align}
The following proposition follows from \cite[Propositions 24.2.12 and 24.2.13]{UrbanskiRoyMunday}:
\begin{prop}\label{prop local property near inddiferent}
    For any small number $r>0$ there exists a constant $C\geq 1$ such that for all $n\in\mathbb{N}$, $\omega\in \Omega(\rational)$, $0\leq k\leq \rate_\omega-1$ and $z\in S_\omega^k(x(\theta),\theta)\setminus B(\omega,r)$ we have 
    \begin{align*}
   \frac{1}{C} \leq 
\frac{|R_\omega^{-n}(z)-\omega|}{n^{-1/\rate_\omega}}\leq C        
\text{ and }
   \frac{1}{C} \leq 
\frac{|(R_\omega^{-n})'(z)|}{n^{-(\rate_\omega+1)/\rate_\omega}}\leq C .    
    \end{align*}
\end{prop}

By Fatou's flower theorem, we have the following (see, for example, \cite[Theorem 24.2.24]{UrbanskiRoyMunday}): For two sets $A,B\subset \mathbb{C}$ we define $\text{dist}(A,B):=\inf\{|z-z'|:z\in A,z'\in B\}$. If $A=\{z\}$, we write $\text{dist}(\{z\},B)=\text{dist}(z,B)$. For a set $A\subset \mathbb{C}$ we denote by $\partial A$ the Euclidean boundary of $A$.
\begin{thm}\label{thm fatou}
There is 
$    0<\beta(\rational)<\min\left\{\smallnumber, 
    \min_{\omega\in \Omega(\rational)}
    \text{dist}\left(\omega,\partial U(\omega,\smallnumber)\right)\right\}
$
    such that for all $\omega\in \Omega(\rational)$ we have 
    \begin{align*}
        J \cap B(\omega, \beta(\rational)) \setminus \{\omega\} \subset  \bigcup_{k=0}^{\rate_\omega-1} S_\omega^k(x(\theta),\theta).
    \end{align*}
\end{thm}
In the following,  $\beta(\rational)$ denotes the number obtained by the above theorem.

\subsection{Construction of a Markov partition}
For $A\subset \mathbb{C}$ we denote by $\text{Int}(A)$ the Euclidean interior of $A$ and by $\overline{A}$ the Euclidean closure of $A$.
For a countable set $\edge$ a partition $\{A_{s}\}_{s\in \edge}$ of $J$ is called a Markov partition of $\rational$ if the following conditions hold:
\begin{itemize}
    \item We have $\bigcup_{k\in \edge}A_k= J$. Moreover, for each $k\in \edge$ we have $A_k=\overline{\text{Int}(A_k)}$.
    \item For each $k,\tilde k\in \edge$ if $\text{Int}(A_k)\cap \text{Int}\rational(A_{\tilde k})\neq \emptyset$ then we have $A_{k} \subset \rational(A_{\tilde k})$. Moreover, for each $k\in \edge$ the map $\rational|_{A_k}$ is injective.
    \item $\{\text{Int}(A_k):k\in \edge\}$ are pairwise disjoint.
\end{itemize}
Let $\expansive<\beta(\rational)/2$ be a expansive constant of the parabolic rational map $\rational$ satisfying the following condition: for every $\omega, \omega'\in \Omega(\rational)$ with $\omega\neq \omega'$ we have 
\begin{align}\label{eq separation}
    \text{dist}(\rational( B(\omega,2\expansive)),\Omega(\rational)\setminus\{\omega\})>2\expansive \text{ and }
    \text{dist}(B(\omega,2\expansive),B(\omega',2\expansive))>0.
\end{align}
Note that since $\rational$ is parabolic, $J \cap \text{Crit}(\rational)=\emptyset$ (see \cite[Theorem 26.2.3]{UrbanskiRoyMunday}). Moreover, by \cite[Theorem 26.2.4]{UrbanskiRoyMunday}, we have 
\begin{align}\label{eq accumulation}
\overline{\text{PC}(\rational)}\cap J=\Omega(\rational), 
\text{ where } \text{PC}(\rational)={\bigcup}_{n\in\mathbb{N}} \rational^{n}(\text{Crit}(\rational))
\end{align}
Hence, 
\begin{align*}
d_\expansive:=\text{dist}\left(J,\left(\overline{\text{PC}(\rational)}\cup \text{Crit}(\rational)\right)\setminus B(\Omega(\rational),\expansive)\right)>0,    
\end{align*}
where $B(\Omega(\rational),\expansive):=\{z\in\mathbb{C}:\text{dist}(z,\Omega(\rational))<\expansive\}$.

Choose $\hat q>0$ sufficiently small so that, for every $\omega\in\Omega(\rational)$ and every $z\in\rational^{-1}(\omega)$,
\begin{align}\label{eq biholomorphic one iteration}
\text{the restriction $\rational:B(z,\hat q)\rightarrow \rational(B(z,\hat q))$ is biholomorphic.
}
\end{align}
Let $z_1$ and $z_2$ be periodic points of $\rational$ with prime periods at least $3$ and disjoint forward orbits. We denote by $O_+(z_i)$ ($i=1,2$) the forward orbit of $z_i$. 
We fix a sufficiently small $r>0$ satisfying
\begin{align}\label{eq How small}
r<\min \{d_\expansive,v,\text{dist}(O_+(z_1),O_+(z_2)),q\}/4.
\end{align}
Since $2r<\min\{d_\expansive, v\}$, for each $z\in J\setminus B(\Omega(\rational),2\expansive)$ we have 
\begin{align}\label{eq disjoint PC Crit}
\left(\text{PC}(\rational)\cup\text{Crit}(\rational)\right)\cap B(z, 2r)=\emptyset.
\end{align}
In particular, by \cite[Corollary 22.5.4]{UrbanskiRoyMunday} for each $z\in J\setminus B(\Omega(\rational),2\expansive)$, $n\in\mathbb{N}$, $w\in \rational^{-n}(B(z, 2r))$ and connected component $U_w$ of $\rational^{-n}(B(z, 2r))$ with $w\in U_w$ the map $\rational^n:U_w\rightarrow B(z, 2r)$ is a conformal homeomorphism. Therefore, the analytic inverse branch $\rational^{-n}_w:B(z, 2r)\rightarrow U_w$ of $\rational^n:U_w\rightarrow B(z, 2r)$ is well defined. Moreover, since $r/2<\text{dist}(O_+(z_1),O_+(z_2))/8$, \cite[Lemma 24.1.13]{UrbanskiRoyMunday} implies that 
\begin{align}\label{eq inverse branch expanding}
    \lim_{n\to\infty}\inf \{|(\rational^n)'(w)|:z\in J\setminus B(\Omega(\rational),2\expansive), w\in \rational^{-n}(B(z, 2r))\}=\infty
\end{align}
In particular,  there exists $\eta>0$ such that for all $z\in J\setminus B(\Omega(\rational),2\expansive)$, $n\in\mathbb{N}$ and $w\in \rational^{-n}(B(z, 2r))$ we have
\begin{align}\label{eq def eta}
    \text{diam}(\rational^{-n}_w(B(z,\eta)))<\expansive/2,
\end{align}
where $\text{diam}(A)$ denotes the Euclidean diameter of $A\subset \mathbb{C}$.

By \cite[Theorem 19.2.2]{Urbanskinoninvertible}, we have the following version of the Koebe distortion theorem: 
there exist a small number $0<\Koebe<1/2$ and a constant $C>0$ such that if $x\in \mathbb{C}$, $\varepsilon>0$ and $f:B(x,\varepsilon)\rightarrow \mathbb{C}$ is a conformal homeomorphism function then for every $y\in B(z,\Koebe \varepsilon)$ we have 
\begin{align}\label{eq Koebe}
    |f'(y)-f'(x)|\leq  C|f'(x)||y-x|.
\end{align}
The following proposition was shown by Denker and Urba\'nski \cite[Lemma 4.4]{DenkerUrbanskiforummath} (see also \cite[Lemma 9.2]{AaronsonDenkerUrbanski}):
\begin{prop}
There exists a Markov partition $\{A_1,\cdots, A_s\}$ such that for each $1\le k\leq s$ we have $\text{diam}(A_k)\leq \Koebe\min\{r,\eta\}$. 
\end{prop}
We define
$    \edge:=\{1,\cdots, s\}.$
We define the incidence matrix $A:\edge\times\edge\rightarrow\{0,1\}$ by setting $A_{\tau,\tau'}=1$ if and only if $\rational(\text{Int}(A_\tau))\cap \text{Int}(A_{\tau'})$.
We set
\[
\MS:=\{\tau\in \edge^{\mathbb{N}\cup\{0\}}: A_{\tau_{n-1},\tau_n}=1,\  n\in\mathbb{N}\}.
\]
We denote by $\MS^n$ the set of all admissible
words of length $n\in\mathbb{N}$ with respect to $A$ and for $\tau\in \MS^n$ we set $|\tau|=n$. 
We denote
by $\MS^{*}$ the set of all admissible words which have a finite length
(i.e. $\MS^*=\cup_{n=1}^\infty\MS^{n}$).
For $n,l\in\mathbb{N}$, $\tau\in \MS^{n}$ and $\tau'\in \MS^{l}$ we define $\tau\tau':=\tau_0\cdots\tau_{n-1}\tau'_{0}\cdots\tau'_{l-1}\in \edge^{n+l}$.
For $\tau\in \MS^{n}$ ($n\in\mathbb{N}$) we define the cylinder set of $\tau$ by
$[\tau]:=\{\tau'\in\MS:\tau_{i}'=\tau_{i},0\leq i\leq n-1\}.$
We endow $\MS$
with the shift metric $d$ which is for $\tau,\tau'\in \MS$ defined by 
$
d(\tau,\tau')=\exp(-\inf\{k\ge 0:\,\, \tau_k \neq \tau_k' \}).
$
For $\tau\in \MS^*$ we define 
\[
A_\tau=A_{\tau_0}\cap\rational ^{-1}(A_{\tau_1})\cap\cdots \cap\rational ^{-(|\tau|-1)}(A_{\tau_{|\tau|-1}})
\]
By Proposition \ref{prop local property near inddiferent} and  \eqref{eq inverse branch expanding}, we have 
\begin{align}\label{eq uniform decay of cylinders}
\lim_{n\to\infty}\max\left\{\text{diam} (A_\tau):\tau\in \MS^{n} \right\}=0.
\end{align}
By \eqref{eq uniform decay of cylinders}, the coding map $\coding:\MS\rightarrow J$ defined by 
$
\{\coding(\tau)\}=\bigcap_{n=1}^{\infty}A_{\tau_0\cdots\tau_{n-1}}
$
is well defined. Moreover, $\coding$ is continuous and surjective and satisfies $\rational\circ \coding=\coding \circ \shift$. Here, $\shift$ denotes the left shift map on $\MS$.
We define 
\[
H:=\{h\in \edge: A_h\cap\mathbb{C}\setminus B(\Omega(\rational),2\expansive)\} \text{ and }\inddiferent:=\edge\setminus H
\]
For $h\in H$ we fix a point $z_h\in A_h$ with $z_h\in J\setminus B(\Omega(\rational),2\expansive)$.
The following lemma follows from the argument in the proof of \cite[Lemma 4.4]{DenkerUrbanskiforummath}. However, for the convenience of the reader, we include a proof. Let 
\begin{align}\label{eq def kappa}
\kappa:=\max\{ \text{diam}(A_e):e\in \edge\}
\end{align}
\begin{lemma}\label{lemma component}
For each $h\in H$, $n\in\mathbb{N}$ and $\tau\in \MS^n$ with $\tau h\in \MS^*$ there exists a unique connected component $\comp(\tau h)$ of $\rational^{-n}(B(z_h, 2\kappa))$ such that 
$
A_{\tau h}\subset \comp(\tau h).
$
\end{lemma}
\begin{proof}
    We fix $w\in A_{\tau h}$. Let $U_w$ be the connected component of $\rational^{-n}(B(z_h, 2\kappa))$ with $w\in U_w$. Recall that $2\kappa<r$. Thus, by \eqref{eq How small}, the analytic inverse branch $\rational^{-n}_w:B(z_h,2\kappa)\rightarrow U_w$ of the map $\rational^n:U_w\rightarrow B(z_h,2\kappa)$ is well defined. Let $x\in A_{\tau h}$. Then $\rational^n(x)\in A_h\subset B(z_h,2\kappa)$. Thus, there exists $y\in U_w$ such that $\rational^n(x)=\rational^n(y)$. The result follows once we show that $x=y$. For a contradiction we assume that $x\neq y$. Then there exists $0\leq l'<n$ such that $\rational^{l'}(x)\neq \rational^{l'}(y)$ and $\rational^{l'+1}(x)=\rational^{l'+1}(y)$. Since $\expansive$ is a expansive constant of $\rational$, there exists $0\leq l\leq l'$ such that 
    \begin{align}\label{eq uniquness of connected component}
    |\rational^l(x)-\rational^l(y)|\geq \expansive.    
    \end{align}
    Note that we have $\rational^l(y)=\rational^{l}(\rational^{-n}_w(\rational^n(y)))\in \rational^l\circ\rational_w^{-n}(B(z_h,2\kappa))$ and $\rational^l(w)=\rational^{l}(\rational^{-n}_w(\rational^n(w)))\in \rational^l\circ\rational_w^{-n}(B(z_h,2\kappa))$. Since $2\kappa<\eta$ and $\rational^l\circ\rational_w^{-n}$ is an analytic inverse branch, \eqref{eq def eta} implies that $|\rational^l(y)-\rational^l(w)|<\expansive/2$. On the other hand, we have  $\rational^{l}(w),\rational^l(x)\in A_{\tau_l}$ and thus, $|\rational^l(w)-\rational^l(x)|\leq \kappa<\expansive/2$. Consequently, we obtain $|\rational^l(x)-\rational^l(y)|<\expansive$. This is a contradiction.
\end{proof}

For each $h\in H$, $n\in\mathbb{N}$ and $\tau\in \MS^n$  with $\tau h\in \MS^*$  we denote by $\comp(\tau h)$ the unique connected component of $\rational^{-n}(B(z_h, 2\kappa))$ obtained in Lemma \ref{lemma component}.
We denote by $\rational^{-n}_{\tau h}$ the inverse of the conformal homeomorphism $\rational^n:\comp(\tau h)\to B(z_h,2\kappa)$.

The following distortion lemma follows immediately from \eqref{eq disjoint PC Crit} and \eqref{eq Koebe}:
\begin{lemma}\label{lemma distortion}
    There exists a constant $C>0$ such that for all $n\in\mathbb{N}$, $h\in H$, $\tau\in \MS^n$ with $\tau h\in \MS^*$ and $x,y\in B(z_{h},2\kappa)$ we have
    \begin{align*}
        |(\rational_{\tau h}^{-n})'(y)-(\rational_{\tau h}^{-n})'(x)|\leq C|\left(\rational_{\tau h}^{-n}\right)'(x)||x-y|.
    \end{align*}
Moreover, by increasing $C$ if necessary, we have  $|(\rational_{\tau h}^{-n})'(y)|\leq C|(\rational_{\tau h}^{-n})'(x)|$.
\end{lemma}

\subsection{Induced systems}\label{sec induced systems}
In this section, we construct an induced system. We define
\begin{align*}
    \induedge_1:=H^2\cap\MS^2 \text{ and }\induedge_n:=\{h\parabolic_1\cdots \parabolic_{n-1}h'\in \MS^{n+1}: h,h'\in H,\parabolic_i\in \inddiferent\}
\end{align*}
for $n\geq 2$. We set $\induedge:=\bigcup_{n\in\mathbb{N}}\induedge_n$. We define the incidence matrix $B:\induedge\times\induedge\rightarrow \{0,1\}$ by setting $B_{\tau,\tilde \tau}=1$ if and only if $\tau_{|\tau|-1}=\tilde \tau_0$. We define the countable Markov shift $(\CMS,\indushift)$ by 
\[
\CMS:=\{\tau\in \induedge^{\mathbb{N}\cup\{0\}}: B_{\tau_{n-1},\tau_n}=1, n\in\mathbb{N}\}.
\]
Here, $\indushift$ denotes the left shift map on $\CMS$. We use the same notation for the countable Markov shift $(\CMS,\indushift)$ as for the Markov shift $(\MS,\shift)$. We also define the shift metric $\tilde d$ on $\CMS$ in the same way as the shift metric $d$ on $\MS$.

Let $\indudomain:=\{[h]:h\in H\}$. We define the first return time function $\return^*:\indudomain\rightarrow\mathbb{N}\cup\{\infty\}$ by 
\[
\return^*(\tau):=\inf\{n\in\mathbb{N}: \shift^n(\tau)\in \indudomain\}.
\]
The first return map $\indushift^*:\{\return^*<\infty\}\rightarrow \indudomain$ of $\shift$ to $\indudomain$ is defined by $\indushift^*(\tau):=\shift^{\return^*(\tau)}(\tau)$. We set $\indudomain:=\bigcap_{n\in \mathbb{N}\cup\{0\}}(\indushift^*)^{-n}(\{\return^*<\infty\})$. 

The following map allows us to identify $(\indudomain,\indushift^*)$ with the countable Markov shift $(\CMS,\indushift)$.
Recall that for $\tau,\tilde \tau\in \induedge$, $B_{\tau,\tilde \tau}=1$ if and only if $\tau_{|\tau|-1}=\tilde \tau_0$.
We define the map $\Proj:\CMS\rightarrow \indudomain$ by 
\[
\Proj(\tau)=h_0p_{0,1}\cdots p_{0,|\tau_0|-1}h_1p_{1,1}\cdots p_{1,|\tau_1|-1}h_2p_{2,1}\cdots p_{2,|\tau_2|-1}\cdots, 
\]
where $\tau=\tau_0\tau_1\cdots$ and $\tau_k=h_kp_{k,1}\cdots p_{k,|\tau_k|-1}h_k'$. 
Here, $p_{k,0}$ is understood to be the empty word.
Note that $\Proj$ is bijection and $\Proj\circ \indushift=\indushift^*\circ \Proj$. Moreover, $\Proj$ and $\Proj^{-1}$ are continuous. Hence, $\Proj$ is a topological conjugacy between $(\indudomain,\indushift^*)$ and $(\CMS,\indushift)$. In particular, for any $\indumeasure\in M(\indushift)$ and Borel measurable function $\indupote$ on $\CMS$ we have 
\begin{align}\label{eq Projection equivalence 1}
    h(\indumeasure)=h(\indumeasure^*) \text{ and }\indumeasure( \indupote )=\indumeasure^*( \indupote^* ),
\end{align}
where $\indupote^*=\indupote\circ \Proj^{-1}$ and $\indumeasure^*=\indumeasure\circ \Proj^{-1}$. Conversely, 
 for any $\indumeasure^*\in M(\indushift^*)$ and Borel measurable function $\indupote^*$ on $\indudomain$ we have 
\begin{align}\label{eq Projection equivalence 2}
    h(\indumeasure)=h(\indumeasure^*) \text{ and }\indumeasure( \indupote )=\indumeasure^*( \indupote^* ),
\end{align}
where $\indupote=\indupote^*\circ \Proj$ and $\indumeasure=\indumeasure^*\circ \Proj$.
We also define the map $\Proj^*:\CMS^*\rightarrow \MS^*$ by 
\[
\Proj^*(\tau)=h_0p_{0,1}\cdots p_{0,|\tau_0|-1}\cdots 
h_{|\tau|-1}p_{|\tau|-1,1}\cdots p_{|\tau|-1,|\tau_{|\tau|-1}|-1}h'_{|\tau|-1}, 
\]
where $\tau=\tau_0\tau_1\cdots\tau_{|\tau|-1}$ and $\tau_k=h_kp_{k,1}\cdots p_{k,|\tau_k|-1}h_k'$. 
Here, $p_{k,0}$ is understood to be the empty word.

The following lemma is an immediately consequence of \eqref{eq inverse branch expanding}: Let 
\[
\return:=\return^*\circ\Proj.
\]
For a function $\phi:J\rightarrow \mathbb{R}$ we define the function $\indupote:\CMS\rightarrow \mathbb{R}$ by  
\begin{align}\label{eq def indupote}
\indupote(\tau):=\sum_{k=0}^{\return(\tau)-1}\phi\circ\rational^{k}(\coding\circ\Proj(\tau))
\end{align}
If $\phi=\log |\rational'|$ then we write $\log |\indumap'|:=\indupote$.
\begin{lemma}\label{lemma uniformly expanding}
For any $s>1$
    there exists $N\in\mathbb{N}$ such that for all $n\geq N$, $\tau\in \CMS^n$ and $x\in \comp(\Proj^*(\tau))$ we have $|(\rational^{|\Proj^*(\tau)|-1})'(x)|>s$. In particular, For all $\tau\in \CMS$ we have 
    $S_N(\log |\indumap'|)(\tau)>s.$
\end{lemma}

By Lemma \ref{lemma uniformly expanding} we obtain the following:
\begin{lemma}\label{lemma exponential decay}
There exist $C>0$ and $\beta>0$ such that for all $n\in\mathbb{N}$ and $\tau\in \CMS^n$ we have $\text{diam}(\comp(\Proj^*(\tau)))\leq Ce^{-\beta n}$.    
\end{lemma}

\section{Thermodynamic formalism for the countable Markov shift $(\CMS,\indushift)$}\label{sec Thermodynamic formalism for the countable Markov shift}
In this section, we describe some properties of the countable Markov shift $(\CMS,\indushift)$. These results will be used in later sections.

 For a continuous function $\tilde\psi$ on $\CMS$ the topological pressure of $\tilde\psi$ introduced by Mauldin and Urba\'nski \cite{mauldin2003graph} is given as
\begin{align}\label{eq def topological induced pressure}
\tilde{P}(\tilde\psi):=\lim_{n\to\infty}\frac{1}{n}\log\sum_{a\in \CMS^n}\exp\left(
\sup_{\tau\in[a]} 
S_n(\tilde\psi)(\tau)
\right).
\end{align}
Let $\phi:J\rightarrow\mathbb{R}$ be a continuous function and let $\indupote$ be the induced potential defined by \eqref{eq def indupote}. 
For some $\beta>0$ the induced potential $\indupote$ is said to be locally H\"older with exponent $\beta$ if  
$    \sup_{n\in\mathbb{N}}\sup_{a\in \CMS^n}\sup\{|\indupote(\tau)-\indupote(\tau')|(\tilde d(\tau,\tau'))^{-\beta}:\tau,\tau'\in [a],\ \tau\neq\tau'\}<\infty.
$
 The induced potential $\indupote$ is said to be locally H\"older if there exists $\beta>0$ such that $\indupote$ is locally H\"older with exponent $\beta$. 
 Lemmas \ref{lemma distortion}, \ref{lemma uniformly expanding}, and \ref{lemma exponential decay} allow us to apply the arguments used in the proofs of \cite[Lemmas 19.3.4 and 19.4.1]{Urbanskinoninvertible}, yielding the following. 
\begin{lemma}\label{lemma geometric Holder}
The function $\log |\indumap'|$ is locally H\"older.
\end{lemma}
In what follows, we fix 
$\phi\in \mathcal{H}$ and a exponent $\exponent>0$ such that $\indupote$ and $\log|\indumap'|$ are locally H\"older with exponent $\beta$. 
Note that, since $\return$ is constant on $[\tau]$ for each $\tau\in \induedge$, $\return$ is locally H\"older with $\exponent$. 

By \cite[Lemma 9.3]{AaronsonDenkerUrbanski}, without loss of generality there are two (repulsive) periodic points with relatively periods whose orbits are outside of $B(\Omega(\rational),2\expansive)$ and in the interior of sets $\{A_k:k\in \edge\}$. 
Thus, by exactly the same argument as in the proof of \cite[Theorem 7.1]{AaronsonDenkerUrbanski}, the countable Markov shift $(\CMS,\indushift)$ is topologically mixing, that is, for all $\tau,\tilde \tau\in \induedge$ there exists $N_{\tau,\tilde \tau}\in\mathbb{N}$ such that for all $n\geq N_{\tau,\tilde \tau}$ we have $[\tau]\cap \indushift^{-n}([\tilde \tau])\neq \emptyset$. Moreover, it is not difficult to show that it satisfies the big images and preimages (BIP) property, that is, there exists a finite set $F\subset \induedge$ such that for all $a\in \induedge$ there exists $b,\tilde b\in F$ such that $B_{b, a}B_{a,\tilde b}=1$. By \cite{Sarig}, these properties imply that 
$\CMS$ is finitely primitive, that is, there exist $n\in\mathbb{N}$ and a finite set $\tilde \Lambda\subset \CMS^n$ such that for all $e,e'\in \induedge $ there is $a=a(e,e')\in \tilde \Lambda$ for which $e ae'\in \CMS^*$. 

These conditions enable us to apply results from the thermodynamic formalism for the countable Markov shift ($\CMS,\indushift$) as  presented in \cite[Section 2]{mauldin2003graph} (see also \cite[Section 17, 18 and 20]{Urbanskinoninvertible}).
Define the pressure function $p:\mathbb{R}^3\rightarrow \mathbb{R}$ by 
\[
\indupressure(b,q,s):=\tilde{P}(-q\indupote-b\log|\indumap'|-s\return).
\]
\begin{thm}{\cite[Theorem 2.1.8]{mauldin2003graph}}\label{thm variational principle induce} For all $(b,q,s)\in \mathbb{R}^3$ we have
\[
\indupressure(b,q,s)=\sup_{\indumeasure}\{h(\indumeasure)+
\indumeasure(-q\indupote-b\log|\indumap'|-s\return) \},
\]
where the supremum is taken over the set of measures $\indumeasure\in M(\indushift)$ satisfying $\indumeasure(-q\indupote-b\log|\indumap'|-s\return)>-\infty$. 
\end{thm}
Let 
\[
Fin:=\{(b,q,s)\in \mathbb{R}^3:\indupressure(b,q,s)<\infty\}.
\]
\begin{prop}{\cite[Proposition 2.1.9]{mauldin2003graph}}\label{prop finite pressure}
 $(b,q,s)\in Fin$ if and only if 
 \[
 \sum_{\tau\in \induedge}\exp\left(\sup_{\tilde\tau\in [\tau]}\left\{\left(-q\indupote-b\log|\indumap'|-s\return\right)(\tau)\right\}\right)<\infty.
 \]
\end{prop}
For $(b,q,s)\in Fin$ a measure $\indumeasure\in M(\indushift)$ is called a Gibbs measure for $-q\indupote-b\log|\indumap'|-s\return$ if there exists a constant $Q\geq 1$ such that for every $n\in\mathbb{N}$, $\tau\in \CMS^n$ and $\tilde\tau\in[\tau]$ we have
\begin{align}\label{eq gibbs}
\frac{1}{Q}\leq \frac{\indumeasure([\tau])}{\exp(S_n(-q\indupote-b\log|\indumap'|-s\return)(\tilde \tau)-\indupressure(b,q,s)n)}\leq Q. 
\end{align}
\begin{thm}{
\cite[Theorem 2.2.4 and Corollary 2.7.5]{mauldin2003graph}}\label{thm Gibbs induce}
    For $(b,q,s)\in Fin$ there exists the unique Gibbs measure $\indumeasure_{b,q,s}\in M(\indushift)$ for $-q\indupote-b\log|\indumap'|-s\return$. Moreover, $\indumeasure_{b,q,s}$ is ergodic.
\end{thm}
In the following, for $(b,q,s)\in Fin$ we denote by $\indumeasure_{b,q,s}$ the unique $\indushift$-invariant Gibbs measure obtained in Theorem \ref{thm Gibbs induce}.
For $(b,q,s)\in \mathbb{R}^3$ we say that $\indumeasure\in M(\indushift)$ is an equilibrium measure for $-q\indupote-b\log|\indumap'|-s\return$ if we have $\indumeasure(-q\indupote-b\log|\indumap'|-s\return)>-\infty$ and $\indupressure(b,q,s)=h(\indumeasure)+\indumeasure(-q\indupote-b\log|\indumap'|-s\return)$.
\begin{thm}{\cite[Theorem 2.2.9]{mauldin2003graph} 
}
\label{thm equilibrium state induce}
  Let  $(b,q,s)\in Fin$. If $\indumeasure_{b,q,s}(-q\indupote-b\log|\indumap'|-s\return)>-\infty$ then $\indumeasure_{b,q,s}$ is the unique equilibrium measure for $-q\indupote-b\log|\indumap'|-s\return$.
\end{thm}
For $\indumeasure\in M(\indushift)$ we define $\lambda(\indumeasure)=\indumeasure( \log|\indumap'|)$
.
\begin{thm}{\cite[Theorem 2.6.12 and Proposition 2.6.13]{mauldin2003graph} (see also \cite[Theorem 20.1.12]{Urbanskinoninvertible})} \label{thm regularity of induced pressure}
The function $(b,q,s)\mapsto\indupressure(b,q,s)$ is real-analytic on $\text{Int}(Fin)$. Moreover, we have Ruelle's formula, that is, for $(b,q,s)\in \text{Int}(Fin)$,
$
\frac{\partial}{\partial b}\indupressure(b,q,s)=-\lambda(\indumeasure_{b,q,s}),
$
$
\frac{\partial}{\partial q}\indupressure(b,q,s)=-\indumeasure_{b,q,s}(\indupote),
$
$
\frac{\partial}{\partial s}\indupressure(b,q,s)=-\indumeasure_{b,q,s}(\return).
$
\end{thm}

\section{Thermodynamic formalism for parabolic rational maps}\label{sec Thermodynamic formalism for parabolic rational maps}
In this section, we develop the thermodynamic formalism for the parabolic rational maps. Throughout this section, we fix a parabolic rational map $\rational$ with $\text{deg}(\rational)\geq 2$ such that $\infty\notin J$ and \eqref{eq assumption simple} holds. We also fix a potential $\phi\in \mathcal{H}$. For simplicity of notation, we write $\Omega:=\Omega(\rational)$

For a continuous function $\psi$ on $J$ we define the topological pressure of $\psi$ by 
\[
P(\psi):=\sup\{h(\mu)+\mu(\psi):\mu\in M(\rational)\}.
\]
For a continuous function $\psi^*$ on $\MS$ we define the topological pressure of $\psi^* $ by 
\[
P^*(\psi^*):=\sup\{h(\mu)+\mu(\psi^*):\mu\in M(\shift)\}.
\]
We define the pressure functions $\pressure, \codingpressure:\mathbb{R}^2\rightarrow \mathbb{R}$ by
\[
\pressure(b,q):=P(-q\phi-b\log|\rational'|) \text{ and }\codingpressure(b,q):=P^*(-q\phi^*-b\log|(\rational^*)'|),
\]
where $\phi^*:=\phi \circ \coding$ and $\log|(\rational^*)'|:=\log |R'\circ\coding|$.
\begin{rem}\label{rem coding pressure}
Recall that the coding map $\coding:\MS\rightarrow J$ is continuous and surjective and satisfies $\rational\circ \coding=\coding \circ \shift$. Moreover, $\coding|_{\coding^{-1}(J\setminus\bigcup_{n=0}^\infty \rational^{-n}(\bigcup_{k\in\edge}\partial A_k))}$ is injective (see, for example, \cite[Theorem 4.5.7]{PrzytyckiUrbanski}).
By \cite[Proposition 2.13]{Bowen}, for any continuous function $\psi$ on $J$ we have 
\begin{align}\label{eq pressure ine}
P(\psi)\leq P^*(\psi\circ \coding).    
\end{align}
By exactly the same argument as in the proof of  \cite[Lemma 4.3]{Bowen}, for any $\mu\in M(\rational)$ there exists $\mu^*\in M(\shift)$ such that $\mu=\mu^*\circ\coding^{-1}$.  
In particular, by \eqref{eq pressure ine}, for a continuous function $\psi$ on $J$ if there exists an equilibrium measure $\codmeas\in M(\shift)$ for $\psi\circ \coding$ such that $h(\codmeas)=h(\codmeas\circ\coding^{-1})$ then $\codmeas\circ\coding^{-1}$ is an equilibrium measure for $\psi$ and $P(\psi)=P^*(\psi\circ\coding)$. Moreover, if $\codmeas\in M(\shift)$ is the unique equilibrium measure for $\psi\circ \coding$ such that $h(\codmeas)=h(\codmeas\circ\coding^{-1})$ then $\codmeas\circ\coding^{-1}$ is the unique equilibrium measure for $\psi$.
 By \cite[Theorem 4.5.9]{PrzytyckiUrbanski}, if $\mu\in M(\shift)$ is ergodic and positive on open sets then we have 
\begin{align}\label{eq entropy presurving}
    h(\mu)=h(\mu\circ\coding^{-1}). 
\end{align}
\end{rem}

For a function $\tilde\psi$ on $\CMS$ and $e\in \edge$ we set $\tilde \psi([e]):=\sup_{\tau\in[e]}\tilde\psi(\tau)$.
We define, for $\omega\in \Omega$ and $q\in \mathbb{R}$, 
\begin{align*}
\alpha_\omega:=\phi(\omega),\ 
\LB:= \max_{\omega\in \Omega}\{-q\alpha_\omega\} \text{ and }
\mathcal{N}:=\left\{(b,q)\in \mathbb{R}^2:p^*(b,q)>\LB\right\}.
\end{align*}
For $\omega\in \Omega$ we define 
\[
\inddiferent_{\omega}:=\{\parabolic\in \inddiferent:A_\parabolic\subset B(\omega,2\expansive)\}.
\]
Then by \eqref{eq separation}, we have $\inddiferent=\bigcup_{\omega\in \Omega}\inddiferent_\omega$. Moreover, if $A_{\parabolic,\tilde \parabolic}=1$ for some $\parabolic,\tilde \parabolic\in \inddiferent$ then there exists $\omega\in \Omega$ such that $\parabolic,\tilde \parabolic\in \inddiferent_\omega$. Thus, for each $n\geq 2$ we have 
\begin{align*}
\induedge_n=\bigcup_{\omega\in \Omega}\induedge_{n,\omega}, \text{ where }
    \induedge_{n,\omega}:=\{h\parabolic_1\cdots \parabolic_{n-1}h'\in \induedge_{n}:\parabolic_i\in \inddiferent_\omega,1\leq i\leq n-1\}.
\end{align*}
By Proposition \ref{prop local property near inddiferent}, there exists a constant $C\geq 1$ such that for all $\omega\in \Omega$, $n\geq 2$ and $e\in \induedge_{n,\omega}$ we have 
\begin{align}\label{eq growth rate}
    C^{-1} n^{-\frac{\rate_\omega+1}{\rate_\omega}}\leq e^{-\log|\indumap'|([e])}\leq Cn^{-\frac{\rate_\omega+1}{\rate_\omega}}
\end{align}
 By proposition \ref{prop local property near inddiferent} and Fatou's flower theorem (Theorem \ref{thm fatou}), for all $\omega\in \Omega$ and $z\in J\cap B(\omega,2v)$ we have $\lim_{n\to\infty}\rational^{-n}_\omega(z)=\omega$. Therefore, since 
 \begin{align}\label{eq decreasing set}
\bigcup_{\parabolic\in \inddiferent_\omega^{n+1}\cap \MS^{n+1}}[\parabolic] \subset\bigcup_{\parabolic\in \inddiferent_\omega^n\cap \MS^n}[\parabolic]     
 \end{align}
 for all $n\in\mathbb{N}$ and $\omega\in\Omega$, for all $\epsilon>0$ 
 there exists $N\in \mathbb{N}$ such that for all $n\geq N$ and $\omega\in \Omega$ we have
 \begin{align}\label{eq n big near inddiferent}
     \bigcup_{\parabolic\in \inddiferent_\omega^n\cap \MS^n}\coding([\parabolic])\subset B(\omega,\epsilon).
 \end{align}

\begin{lemma}\label{lem finite pressure on a special parameter}
    For all $(b,q)\in\mathbb{R}^2$ and $s\in(\LB,\infty)$ we have $(b,q,s)\in Fin$. 
   In particular, for all $(b,q)\in\mathcal{N}$ we have $(b,q,\codingpressure(b,q))\in Fin$. 
\end{lemma}

\begin{proof}
Fix $(b,q)\in\mathbb{R}^2$ and $s\in( \LB,\infty)$.
We take a small $\epsilon>0$ with $ \LB+\epsilon<s$. 
Since $\phi\circ\coding$ is continuous on $\MS$, \eqref{eq n big near inddiferent} implies that there exists $N\geq 2$ such that for all $\omega\in \Omega$, $n\geq N$, $\tau\in \inddiferent_\omega^n\cap \MS^n$ and $\tilde \tau\in[\tau]$ we have $|q\phi^*(\tilde \tau)-q\alpha_\omega|<\epsilon$. Thus, by \eqref{eq growth rate}, we obtain
\begin{align*}
&\sum_{\omega\in \Omega}
\sum_{n=N+1}^{\infty}
\sum_{\tau\in\induedge_{n,\omega}}
e^{\left(-q\indupote-b\log|\indumap'|- s \return\right)([\tau])}
\\&\ll
\sum_{\omega\in \Omega}
\sum_{n=N+1}^{\infty}
\sum_{\tau\in\induedge_{n,\omega}}
n^{-b\frac{\rate_\omega+1}{\rate_\omega}}
e^{(-q\phi^*\circ\Proj)([\tau])- s}
\\&\exp
\left(
\left(
\sum_{k=1}^{n-N}(-q\phi^*- s )\circ\shift^k\circ\Proj
\right)
([\tau])
\right)
e^{\left(\sum_{k=n-N+1}^{n-1}(-q\phi^*- s )\circ\shift^k\circ\Proj\right)([\tau])}
\\&
\asymp
\sum_{\omega\in \Omega}
\sum_{n=N+1}^{\infty}
\sum_{\tau\in\induedge_{n,\omega}}
\frac{
e^{\left(
\sum_{k=1}^{n-N}(-q\phi^*- s )\circ\shift^k\circ\Proj
\right)([\tau])}}{n^{b\frac{\rate_\omega+1}{\rate_\omega}}}
\leq
\sum_{\omega\in \Omega}\sum_{n=N+1}^{\infty}\frac{
e^{( \LB+\epsilon- s )n
}}{n^{b\frac{\rate_\omega+1}{\rate_\omega}}}<\infty.
\end{align*}
Therefore, by Proposition \ref{prop finite pressure}, we are done.
\end{proof}

For $(b,q)\in\mathbb{R}^2$ with $(b,q,\codingpressure(b,q))\in Fin$ we write $\indumeasure_{b,q}:=\indumeasure_{b,q,\codingpressure(b,q)}$.

\begin{lemma}\label{lemma finite return time}
    For all $(b,q)\in\mathcal{N}$ and $\ell\in\mathbb{N}$ we have $\indumeasure_{b,q}(\return^\ell)<\infty$. Moreover, we have 
    $\indumeasure_{b,q}( |\indupote|^{\ell})<\infty$ and $\indumeasure_{b,q}(|\log|\indumap'||^\ell)<\infty$. In particular, $\indumeasure_{b,q}$ is the unique equilibrium measure for $-q\indupote-b\log|\indumap'|-\codingpressure(b,q)\return$.
\end{lemma}
\begin{proof}
    Let $(b,q)\in\mathcal{N}$ and let $\ell\in\mathbb{N}$.  
We take a small $\epsilon>0$ with $ \LB+2\epsilon<\codingpressure(b,q)$. 
Since $\lim_{x\to\infty}x^\ell e^{-\epsilon x}=0$, there exists $N\geq 2$ such that for all $n\geq N$ we have $n^\ell\leq e^{\epsilon n}$. Therefore, by \eqref{eq gibbs}, we obtain
\begin{align*}
&   \sum_{\omega\in \Omega}
\sum_{n=N}^{\infty}
\sum_{\tau\in\induedge_{n,\omega}}
\int_{[\tau]} \return^\ell d\indumeasure_{b,q}
   = 
   \sum_{\omega\in \Omega}
\sum_{n=N}^{\infty}
\sum_{\tau\in\induedge_{n,\omega}}
n^\ell 
\indumeasure_{b,q}([\tau])
  \\& \asymp
  \sum_{\omega\in \Omega}
\sum_{n=N}^{\infty}
\sum_{\tau\in\induedge_{n,\omega}}
  n^\ell
   \exp\left((-q\indupote-b\log|\indumap'|-\codingpressure(b,q)\return)([\tau])-\indupressure(b,q,\codingpressure(b,q))\right)
   \\&
   \leq 
   e^{-\indupressure(b,q,\codingpressure(b,q))} 
   \sum_{\omega\in \Omega}
\sum_{n=N}^{\infty}
\sum_{\tau\in\induedge_{n,\omega}}
   \exp\left((-q\indupote-b\log|\indumap'|-(\codingpressure(b,q)-\epsilon)\return)([\tau])\right)
\end{align*}
By the same argument in the proof of Lemma \ref{lem finite pressure on a special parameter}, we can show that
\[
\sum_{\omega\in \Omega}
\sum_{n=N}^{\infty}
\sum_{\tau\in\induedge_{n,\omega}}
   \exp\left((-q\indupote-b\log|\indumap'|-(\codingpressure(b,q)-\epsilon)\return)([\tau])\right)<\infty,
\]
which yields $\indumeasure_{b,q}(\return^\ell )<\infty$. 
   Since $\MS$ is compact and $\phi^*$ is continuous, $
   \indumeasure_{b,q}( |\indupote|^\ell)
   \leq \max_{\tau\in\MS}|\phi^*(\tau)|^\ell\indumeasure_{b,q}( \return^\ell) <\infty.
   $
Recall that $\text{Crit}(\rational)\cap J=\emptyset$ and thus, $\log|\rational'|$ is continuous on $J$. 
   Hence,
   by the same argument, we obtain $\indumeasure_{b,q}(|\log|\indumap'||^\ell)<\infty$.
\end{proof}

Define, for a finite set $\mathcal{L}$ and $\{\nu_{\ell}\}_{\ell\in \mathcal{L}}\subset M(\rational)$,
$
\text{Conv}(\{\nu_{\ell}\}_{\ell\in \mathcal{L}}):
=
\{\sum_{\ell\in\mathcal{L}}p_\ell\nu_{\ell}:\{p_\ell\}_{\ell\in\mathcal{L}}\subset[0,1],\ \sum_{\ell\in\mathcal{L}}p_\ell=1
\}.$

\begin{lemma}\label{lemma equivalent condition not liftable}
    Let $\codmeas\in M(\shift)$. We have that $\codmeas(\indudomain)=0$ if and only if $\codmeas\in \conv$, where $\delta_{\tau_\omega}$ denotes the Dirac measure at $\tau_\omega$. 
\end{lemma}
\begin{proof}
    Since $\indudomain\cap\{\tau_\omega\}_{\omega\in \Omega,\tau_\omega\in \coding^{-1}(\omega)}=\emptyset$, if $\codmeas\in \conv$ then $\codmeas(\indudomain)=0$. Suppose that $\codmeas(\indudomain)=0$. Since $\codmeas$ is $\shift$-invariant, we have that
    $\sum_{\omega\in \Omega}\codmeas\bigg(\bigcup_{\parabolic\in \inddiferent_\omega^n\cap \MS^n}[\parabolic]\bigg)=1$ for all $n\in\mathbb{N}$. Moreover, by \eqref{eq decreasing set} and \eqref{eq n big near inddiferent}, we obtain $\lim_{n\to\infty}\codmeas(\bigcup_{\parabolic\in \inddiferent_\omega^n\cap \MS^n}[\parabolic])=\codmeas(\coding^{-1}(\omega))$ 
 for each $\omega\in \Omega$. Hence, $\codmeas\in\conv$.  
\end{proof}

For $\indumeasure\in M(\indushift)$ we define 
\begin{align}\label{eq def lift}
\codmeas:=\frac{1}{\indumeasure(\return)}
\sum_{n=0}^{\infty}\sum_{k=n+1}^\infty \indumeasure\circ\Proj^{-1}|_{\{\return^*=k\}}\circ \shift^{-n}.    
\end{align}
Since $\indushift^*$ is a first return map of $\shift$, for $\indumeasure\in M(\indushift)$ with $\indumeasure( \return )<\infty$ we have $\codmeas\in M(\shift)$ (for example see \cite[Proposition 1.4.3]{viana}). 
Also, by \eqref{eq Projection equivalence 1}, if $\indumeasure\in M(\indushift)$ with $\indumeasure( \return )<\infty$ is ergodic then we have Abramov-Kac's formula  (see \cite[Theorem 2.3]{PesinSenti}):
\begin{align}\label{eq Abramov-Kac's formula}
    \indumeasure(\return)h(\codmeas)=h(\indumeasure) \text{ and }\indumeasure( \return )\codmeas(\psi\circ \coding)=\indumeasure(\tilde\psi)
\end{align}
for a continuous function $\psi$ on $J$, where $\tilde \psi$ is the induced potential of $\psi$ defined by \eqref{eq def indupote}.  Conversely, by \eqref{eq Projection equivalence 2}, for an ergodic measure $\codmeas\in M(\shift)$ with $\codmeas(\indudomain)>0$ and $\indumeasure:=\codmeas|_{\indudomain}\circ\Proj/\codmeas(\indudomain)$  we have 
\begin{align}\label{eq classical Abramov-Kac's formula}
    \indumeasure( \return ) h(\codmeas)=h(\indumeasure) \text{ and }\indumeasure( \return )\codmeas(\psi\circ\coding)=\indumeasure( \tilde\psi)
\end{align}
for a continuous function $\psi$ on $J$, where $\tilde \psi$ is the induced potential of $\psi$ defined by \eqref{eq def indupote}.

For $\codmeas\in M(\shift)$ we define $\lambda(\codmeas):=\codmeas(\log |(\rational^*)'|)$.
\begin{lemma}\label{lemma positive Lyapunov}
    Let $\mu\in M(\rational)$. Then $\lambda(\mu)>0$ if and only if $\mu\notin \text{Conv}(\{\delta_\omega\}_{\omega\in \Omega})$.
\end{lemma}
\begin{proof}
Since $\log |\rational'(\omega)|=0$ for all $\omega\in \Omega$, if $\mu\in \text{Conv}(\{\delta_\omega\}_{\omega\in \Omega})$ then we have $\lambda(\mu)=0$. Let $\mu\notin \text{Conv}(\{\delta_\omega\}_{\omega\in \Omega})$. To show $\lambda(\mu)>0$, by the ergodic decomposition theorem (see \cite[Theorem 5.1.3]{viana}), we may assume that $\mu$ is ergodic. By Remark \ref{rem coding pressure}, there exists $\codmeas\in M(\shift)$ such that $\mu=\codmeas\circ \coding^{-1}$. By Lemma \ref{lemma equivalent condition not liftable}, we have $\mu^*(\indudomain)>0$. Therefore, by \eqref{eq classical Abramov-Kac's formula} and Lemma \ref{lemma uniformly expanding}, we obtain 
 $   \lambda(\mu)=\lambda(\codmeas)=(\indumeasure(\return))^{-1}\lambda(\indumeasure)>0,
$
where $\indumeasure:=\codmeas|_{\indudomain}\circ\Proj/\codmeas(\indudomain)$.
\end{proof}

For $(b,q)\in\mathcal{N}$, we define $\codmeas_{b,q}$ by \eqref{eq def lift} using $\indumeasure_{b,q}$. 
Note that by \eqref{eq def lift} and \eqref{eq gibbs}, $\codmeas_{b,q}$ is positive on open sets. Thus, by \eqref{eq entropy presurving}, for any $(b,q)\in\mathcal{N}$ we have 
\begin{align}\label{eq entropy preserving Gibbs}
    h(\codmeas_{b,q})=h(\codmeas_{b,q}\circ \coding^{-1})
\end{align}
By Lemma \ref{lemma finite return time} and \eqref{eq Abramov-Kac's formula}, we obtain
\begin{align}\label{eq induced pressure is less than zero}
&\indupressure(b,q,\codingpressure(b,q))=h(\indumeasure_{b,q})+\indumeasure_{b,q}(-q\indupote-b\log|\indumap'|-\codingpressure(b,q)\return)
    \\&
    \nonumber
    =\indumeasure_{b,q}(\return)
    \left(
    h(\codmeas_{b,q})+\codmeas_{b,q}(-q\phi^*-b\log|(\rational^*)'|)-\codingpressure(b,q)
    \right) \leq 0.
\end{align}
Moreover, we obtain the following:
\begin{thm}\label{thm uniquness and existence of the equilibrium state}
     For all $(b,q)\in \mathcal{N}$ we have  $\indupressure(b,q,\codingpressure(b,q))=0$. Furthermore,  for $(b,q)\in \mathcal{N}$, the measure $\codmeas_{b,q}$ is the unique equilibrium measure for $-q\phi^*-b\log |(\rational^*)'|$.
\end{thm}
\begin{proof}
Let $(b,q)\in\mathcal{N}$. We first show that $\indupressure(b,q,\codingpressure(b,q))=0$. By \eqref{eq induced pressure is less than zero}, it is enough to show that $\indupressure(b,q,\codingpressure(b,q))\geq 0$. Let $\epsilon>0$ be a small number with $ \LB+\epsilon<\codingpressure(b,q)$. By \cite[Corollary 9.10.1]{walters2000introduction}, there exists an ergodic measure $\codmeas\in M(\shift)$ such that we have 
\[
h(\codmeas)+\codmeas (-q\phi^*-b\log|(\rational^*)'|)>\codingpressure(b,q)-\epsilon.
\]
Note that $\codmeas\notin \conv$. Indeed, if $\codmeas\in \conv$ then we have 
$h(\codmeas)+\codmeas (-q\phi^*-b\log|(\rational^*)'|)+\epsilon\leq  \LB+\epsilon<\codingpressure(b,q)$ which yields a contradiction. Thus, by Lemma \ref{lemma equivalent condition not liftable} and \eqref{eq classical Abramov-Kac's formula}, we obtain 
\begin{align}\label{eq proof of uniquness existence larger than zero}
    &\indupressure(b,q,\codingpressure(b,q)-\epsilon)\geq h(\indumeasure)+\indumeasure(-q\indupote-b\log |\indumap'|-(\codingpressure(b,q)-\epsilon)\return)
    \\&=\indumeasure( \return )\left(h(\codmeas)+\codmeas (-q\phi^*-b\log |(\rational^*)'|)-(\codingpressure(b,q)-\epsilon)\right)>0, \nonumber
\end{align}
where $\indumeasure:=\codmeas|_{\indudomain}\circ\Proj/\codmeas(\indudomain)$. 
By Lemma \ref{lem finite pressure on a special parameter}, 
the function $s\mapsto\indupressure(b,q,s)$ is continuous on $( \LB,\infty)$. Hence, by \eqref{eq proof of uniquness existence larger than zero} and \eqref{eq induced pressure is less than zero}, we obtain $\indupressure(b,q,\codingpressure(b,q))=0$. Moreover, by \eqref{eq induced pressure is less than zero}, the measure $\codmeas_{b,q}$ is an equilibrium measure for $-q\phi^*-b\log|(\rational^*)'|$.

Next, we shall show that $\codmeas_{b,q}$ is the unique equilibrium measure for $-q\phi^*-b\log|(\rational^*)'|$. Let $\nu^*$ be an equilibrium measure for $-q\phi^*-b\log|(\rational^*)'|$.
By the ergodic decomposition theorem (see \cite[Theorem 5.1.3]{viana}), we may assume that $\nu^*$ is ergodic.  As above, we have $\nu^*\notin \conv$ and thus, $\nu(\indudomain)>0$. 
Let $\tilde \nu:=\nu^*|_{\indudomain}\circ\Proj/\nu^*(\indudomain)$.
Then, $\tilde\nu$ is an equilibrium measure for $-q\indupote-b\log|\indumap'|-\indupressure(b,q)\return$. Indeed, by Theorem \ref{thm variational principle induce}, \eqref{eq classical Abramov-Kac's formula} and \eqref{eq induced pressure is less than zero}, we have 
\begin{align*}
   &0\geq\indupressure(b,q,\codingpressure(b,q))
   \geq h(\tilde\nu)+\tilde\nu(-q\indupote-b\log |\indumap'|-\codingpressure(b,q)\return)
    \\&=\tilde\nu( \return)\left(h(\nu^*)+\nu^*(-q\phi^*-b\log |(\rational^*)'|)-\codingpressure(b,q)\right)=0. \nonumber    
\end{align*}
Therefore, by the uniqueness of the equilibrium measure (Theorem \ref{thm equilibrium state induce}), we obtain $\tilde \nu=\indumeasure_{b,q}$ and thus, $\nu^*=\codmeas_{b,q}$.
\end{proof}

For each $(b,q)\in \mathcal{N}$ we set $\mu_{b,q}:=\codmeas_{b,q}\circ \coding^{-1}$.
Then by Lemma \ref{lemma uniformly expanding} and \eqref{eq Abramov-Kac's formula}, for all $(b,q)\in \mathcal{N}$ we have 
\begin{align}
\lambda(\mu_{b,q})=\lambda(\codmeas_{b,q})=(\indumeasure_{b,q}(\return))^{-1}\lambda(\indumeasure_{b,q})>0.    
\end{align}
Combining Theorem \ref{thm uniquness and existence of the equilibrium state} with Remark \ref{rem coding pressure} and \eqref{eq entropy preserving Gibbs}, we obtain the following:
\begin{thm}\label{thm equilibrium state rational}
    For all $(b,q)\in \mathcal{N}$ the measure $\mu_{b,q}$ is the unique equilibrium measure for $-q\phi-b\log |\rational'|$ with $\lambda(\mu_{b,q})>0$. Moreover, we have $p(b,q)=\codingpressure(b,q)$ and $\indupressure(b,q,p(b,q))=0$. 
\end{thm}
Combining Theorem \ref{thm equilibrium state rational} with \eqref{eq pressure ine}, we obtain
\[
\mathcal{N}=\left\{(b,q)\in \mathbb{R}^2:p(b,q)>\LB\right\}.
\]
By the definition of the topological pressure and \eqref{eq pressure ine}, for $(b,q)\notin\mathcal{N}$ we have $\LB\geq\codingpressure(b,q)\geq p(b,q)\geq-q\alpha_\omega= \LB$, where $\omega\in \Omega$ satisfies $-q\alpha_\omega=\LB$. Combining this with Theorem \ref{thm equilibrium state rational}, for all $(b,q)\in\mathbb{R}^2$ we have 
\begin{align}\label{eq code and noncode pressure}
\pressure(b,q)=\codingpressure(b,q).
\end{align}

For two function $\psi_1,\psi_2:\CMS\rightarrow\mathbb{R}$ and $(b,q)\in \mathcal{N}$ we define the asymptotic variance of $\psi_1$ and $\psi_2$ by
\[
\sigma^2_{b,q}(\psi_1,\psi_2):=    \lim_{n\to\infty}\frac{1}{n}
    \indumeasure_{b,q}\left( S_n\left(\psi_1-\indumeasure_{b,q}(\psi_1) \right)
    S_n\left(\psi_2-\indumeasure_{b,q}(\psi_2)\right)\right)
\]
when the limit exists. If $\psi_1=\psi_2$ then we write $\sigma^2_{b,q}(\psi_1):=\sigma^2_{b,q}(\psi_1,\psi_2)$.


\begin{thm}\label{thm regularity of non induced pressure}
The pressure function $(b,q)\mapsto p(b,q)$ is real-analytic on $\mathcal{N}$ and for $(b,q)\in \mathcal{N}$ we have
    \begin{align}\label{eq ruell's formula noninduced}
    \frac{\partial}{\partial b}p(b,q)=-\lambda(\mu_{b,q}) \text{ and } \frac{\partial}{\partial q}p(b,q)=-\mu_{b,q}( \phi ). 
    \end{align} 
    Moreover, for $(b,q)\in \mathcal{N}$ we have $\frac{\partial^2}{\partial q^2}p(b,q)=0$
    if and only if $\alpha_{\min}=\alpha_{\max}$.  
\end{thm}

\begin{proof}
    Let $(b_0,q_0)\in\mathcal{N}$. By Lemma \ref{lem finite pressure on a special parameter}, there exists an open neighborhood $O\subset \mathbb{R}^3$ of $(b_0,q_0,p(b_0,q_0))$ such that for all $(b,q,s)\in O$ we have $\indupressure(b,q,s)<\infty$. Also, by Theorem \ref{thm regularity of induced pressure}, we have 
$
\left.
\frac{\partial}{\partial s}\indupressure(b,q,s)\right|_{(b,q,s)=(b_0,q_0,p(b_0,q_0))}=-\indumeasure_{b_0,q_0}( \return)<0.
$
Therefore, by the implicit function theorem, Theorem \ref{thm regularity of induced pressure} and Theorem \ref{thm equilibrium state rational}, the function $p$ is real-analytic at $(b_0,q_0)$. Moreover, the implicit function theorem, Ruelle's formula (see Theorem \ref{thm regularity of induced pressure}) and \eqref{eq Abramov-Kac's formula} gives
\begin{align*}
    &\frac{\partial}{\partial b}p(b,q)=\frac{-\lambda(\indumeasure_{b,q})}{\indumeasure_{b,q}( \return)}
    =-\lambda(\mu_{b,q})
    \text{ and }
    \frac{\partial}{\partial q}p(b,q)=\frac{-\indumeasure_{b,q}( \indupote)}{\indumeasure_{b,q}( \return)}
    =-\mu_{b,q}( \phi ).
\end{align*}
Also, by the implicit function theorem, \eqref{eq Abramov-Kac's formula} and Ruelle's formula for the second derivative of the pressure function (see \cite[Proposition 2.6.14]{mauldin2003graph}), we obtain
\begin{align*}
    \frac{\partial^2}{\partial q^2}p(b,q)
    =
\frac{    \sigma^2_{b,q}(\indupote)\left(\indumeasure_{b,q}( \return)\right)^2
    -2\sigma_{b,q}^2
    (\indupote,\return)
    \indumeasure_{b,q}(\indupote)
    \indumeasure_{b,q}( \return)
    +\sigma^2_{b,q}(\return)(\indumeasure_{b,q}(\indupote))^2}
    {(\indumeasure_{b,q}( \return) )^3}.
\end{align*}
By \eqref{eq Abramov-Kac's formula}, we have
\begin{align*}
&    \sigma^2_{b,q}(\indupote)\left(\indumeasure_{b,q}( \return)\right)^2
    -2\sigma_{b,q}^2
    (\indupote,\return)
    \indumeasure_{b,q}( 
    \indupote)\indumeasure_{b,q}
    ( \return )
    +\sigma^2_{b,q}(\return)(\indumeasure_{b,q}(\indupote ))^2
    \\&
    =\lim_{n\to\infty}\frac{1}{n}
    \indumeasure_{b,q}
    \Bigg(\left(
    S_n\left(\indupote-\indumeasure_{b,q}(\indupote )\right)\indumeasure_{b,q}( \return )
    \right.
    \left.\left.
    -S_n\left(\return-\indumeasure_{b,q}(\return )\right)\indumeasure_{b,q}(\indupote )
    \right)^2\right)
    \\&=(\indumeasure_{b,q}(\return))^2
    \sigma^2_{b,q}(\indupote-\codmeas_{b,q}(\phi^*)\return).
\end{align*}
Hence, we obtain
\begin{align}\label{eq important formula second derivative}
    \frac{\partial^2}{\partial q^2}p(b,q)=\frac{\sigma^2_{b,q}(\indupote-\codmeas_{b,q}(\phi^*)\return)}{\indumeasure_{b,q}( \return ) }.
\end{align}
We shall show the last statement. 
If $\alpha_{\text{min}}=\alpha_{\text{max}}$ then there exists a constant $c\in\mathbb{R}$ such that $\{c\}=\{\int \phi d\mu:\mu\in M(\rational)\}$, which implies that $\{c\}=\{\int \phi^* d\codmeas:\codmeas\in M(\shift)\}$.
Therefore, by \eqref{eq ruell's formula noninduced}, we obtain $\frac{\partial^2}{\partial q^2}p(b,q)=0$. 

Conversely, we assume that  $\frac{\partial^2}{\partial q^2}p(b,q)=0$. Then, by \eqref{eq important formula second derivative}, we have $\sigma^2_{b,q}(\indupote-\mu_{b,q}(\phi)\return)=0$. Thus, by 
Lemma \ref{lemma finite return time} and \cite[Lemma 4.8.8]{mauldin2003graph}, there exists 
a
bounded continuous function $\tilde u:\CMS\rightarrow \mathbb{R}$ such that 
\begin{align}\label{eq proof of cohomologous indu coboundary}
    \indupote-\codmeas_{b,q}(\phi^*)\return=\tilde u-\tilde u\circ\indushift.
\end{align}
Then for all $\omega\in \Omega$ we have 
\begin{align}\label{eq proof of cohomologous boundary}
\alpha_\omega=\mu_{b,q}^*(\phi^*) \text{ and thus, }\phi^*(\tau)=\codmeas_{b,q}(\phi^*) \text{ for all $\omega\in \Omega$ and }\tau\in\coding^{-1}(\omega).    
\end{align}
The proof of \eqref{eq proof of cohomologous boundary} proceeds as follows:
For a contradiction we assume that there exists $\omega\in \Omega$ such that $\alpha_\omega=\phi(\omega)<\codmeas_{b,q}(\phi^*)$. Then, since $\phi$ is continuous on $J$, \eqref{eq n big near inddiferent} implies that there exist $N\in\mathbb{N}$ and a small number $\epsilon>0$ such that for all $n\geq N$ and $x\in    \bigcup_{\parabolic\in \inddiferent_\omega^n\cap \MS^n}\coding([\parabolic])$ we have 
$\phi(x)<\codmeas_{b,q}(\phi^*)-\epsilon.$    
 We fix a sequence $\{\tau_k\}_{k\in\mathbb{N}}$ with $\tau_k\in\induedge_{k,\omega}$ ($k\geq 2$). Then, we obtain
\[
\limsup_{k\to\infty}(\indupote-\mu_{b,q}^*(\phi^*)\return)(\tau_k)=\limsup_{k\to\infty}\sum_{j=0}^{k-1}(\phi-\codmeas_{b,q}(\phi^*))(\rational^j(\coding\circ\Proj(\tau_k)))=-\infty.
\]
However, since $\tilde u$ is a bounded function, for all $k\in\mathbb{N}$ we obtain $|(\indupote-\mu_{b,q}(\phi)\return)(\tau_k)|\leq 2\sup\{|\tilde u(\tau)|:{\tau\in \CMS}\}<\infty$. This is a contradiction. By a similar argument, we also obtain a contradiction in the case $\alpha_\omega>\codmeas_{b,q}(\phi^*)$. Hence, we obtain \eqref{eq proof of cohomologous boundary}.

We set 
\[
\mathcal{P}:=\bigcup_{\omega\in\Omega}\coding^{-1}(\omega),\ 
\mathcal{E}:=\bigcup_{\omega\in\Omega}\bigcup_{n\in\mathbb{N}}\shift^{-n}\circ\coding^{-1}(\omega)\setminus\mathcal{P}
\text{ and }
\mathcal{S}:=
\bigcup_{\omega\in\Omega}\bigcup_{\parabolic\in\inddiferent_\omega}[\parabolic]\setminus (\mathcal{P}\cup\mathcal{E}).
\]
Note that we have the direct decomposition 
$\MS=\mathcal{P}\cup\mathcal{E}\cup\mathcal{S}
\cup \indudomain.$
We will inductively construct a Borel bounded function $u^*:\MS\rightarrow \mathbb{R}$ such that for all $\tau\in \mathcal{P}\cup\mathcal{S}\cup\indudomain$ we have $\phi^*(\tau)=u^*(\tau)-u^*\circ\shift(\tau)+\codmeas_{b,q}(\phi^*)$. 
Define 
\begin{align}\label{eq proof of cohomologous construction 1}
    u^*(\tau):=\tilde u(\Proj^{-1}(\tau))
    \text{ for all $\tau\in \indudomain$}
\text{ and }u^*(\tau):=0 \text{ for all }\tau\in \mathcal{P}\cup\mathcal{E}
    \end{align}
Let $\omega\in \Omega$. 
We define, for $n\in\mathbb{N}$, 
\[
I_{\omega,n}:=\bigcup_{\parabolic \in \inddiferent_\omega^n\cap\MS^n}\bigcup_{h\in H, \gamma h\in\MS^{n+1}}[\parabolic h]\setminus \mathcal{E}
\]
Then by \eqref{eq separation}, we have $\bigcup_{\parabolic \in \inddiferent_\omega}[\parabolic]\setminus (\mathcal{P\cup} \mathcal{E})=\bigcup_{n\in\mathbb{N}}I_{\omega,n}$.
Since for $\tau\in I_{\omega,1}$
we have $\shift(\tau)\in \indudomain$, $u^*(\shift(\tau))$ is already defined by \eqref{eq proof of cohomologous construction 1}. 
Thus, the following definition is well-defined:
$u^*(\tau):=u^*(\shift(\tau))+\phi^*(\tau)-\codmeas_{b,q}(\phi^*)
\text{ for } \tau\in I_{\omega,1}.    
$
Let $n\in\mathbb{N}$.
Assume that for all $1\leq \ell\leq n$ and $\tau\in I_{\omega,\ell}$
we have already defined $u^*(\tau)$ by 
$
u^*(\tau)=u^*(\shift(\tau))+\phi^*(\tau)-\codmeas_{b,q}(\phi^*).
$ For $\tau\in I_{\omega,n+1}$ we define $u^*(\tau):=u^*(\shift(\tau))+\phi^*(\tau)-\codmeas_{b,q}(\phi^*)$. Therefore, by induction, the following definition is well-defined:
\begin{align}\label{eq proof of cohomologous def u on P}
    u^*(\tau)=u^*(\shift(\tau))+\phi^*(\tau)-\codmeas_{b,q}(\phi^*) \text{ for }\tau\in \bigcup_{\parabolic \in \inddiferent_\omega}[\parabolic]\setminus (\mathcal{P}\cup \mathcal{E}).
\end{align}
We shall show that $u^*:\MS\rightarrow\mathbb{R}$ defined by \eqref{eq proof of cohomologous construction 1} and \eqref{eq proof of cohomologous def u on P} satisfies 
\begin{align}\label{eq proof of cohomologous coboundary}
    \phi^*(\tau)=u^*(\tau)-u^*\circ\shift(\tau)+\codmeas_{b,q}(\phi^*) \text{ for all }\tau\in \mathcal{P}\cup\mathcal{S}\cup\indudomain.
\end{align} 
By \eqref{eq proof of cohomologous boundary}, \eqref{eq proof of cohomologous construction 1} and \eqref{eq proof of cohomologous def u on P}, $u^*$ satisfies \eqref{eq proof of cohomologous coboundary} on $\mathcal{P}\cup\mathcal{S}$. 
By \eqref{eq proof of cohomologous indu coboundary}, $u^*$ satisfies \eqref{eq proof of cohomologous coboundary} on $\{\return^*=1\}\cap\indudomain$.  Let $n\geq 2$ and let $\tau\in \{\return^*=n\}\cap \indudomain$.
Note that by \eqref{eq separation}, there exists $\omega\in \Omega$ such that for all $1\leq k\leq n-1$ we have $\shift^{k}(\tau)\in I_{\omega,n-k}$.
By \eqref{eq proof of cohomologous indu coboundary}, \eqref{eq proof of cohomologous construction 1} and \eqref{eq proof of cohomologous def u on P}, we have
\begin{align*}
    &\sum_{k=0}^{n-1}(\phi^*-\codmeas_{b,q}(\phi^*))(\shift^k(\tau))=\tilde u (\Proj^{-1}(\tau))-\tilde u(\Proj^{-1}(\shift^n(\tau)))
    \\&= u^*(\tau)-u^*(\shift(\tau))+\sum_{k=1}^{n-1}\left(u^*(\shift^k(\tau))-u^*(\shift^{k+1}(\tau))\right)
    \\&=u^*(\tau)-u^*(\shift(\tau))+\sum_{k=1}^{n-1}(\phi^*-\codmeas_{b,q}(\phi^*))(\shift^k(\tau)).
\end{align*}
Hence, we obtain $\phi^*(\tau)=u^*(\tau)-u^*(\shift(\tau))+\mu_{b,q}^*(\phi^*)$. 
This completes the proof of \eqref{eq proof of cohomologous coboundary}. Note that $\mathcal{P}\cup\mathcal{E}$ is a countable set. Thus, by the definition of $u^*$ and the continuity of $\tilde u$, $u$ is a Borel function. The boundedness follows from the boundedness of $\tilde u$ and $\phi^*$ and \eqref{eq proof of cohomologous coboundary}. Since $\mathcal{E}$ is a countable set and contains no periodic points, for each $\codmeas\in M(\shift)$ we have $\codmeas(\mathcal{E})=0$ and thus, by \eqref{eq proof of cohomologous coboundary}, 
\[
\mu^*(\phi^*)=\int_{\MS\setminus\mathcal{E}}\phi^* d\codmeas
=\int(u^*-u^*\circ \shift)d\codmeas+\mu_{b,q}^*(\phi^*)=\mu_{b,q}^*(\phi^*).
\]
Hence, we obtain $\alpha_{\min}=\alpha_{\max}=\{\mu_{b,q}(\phi)\}$. 
\end{proof}

For $(b,q)\in\mathbb{R}^2$ we denote by $M_{b,q}$ the set of equilibrium measures for $-q\phi-b\log |\rational'|$.

For a convex function $(x_1,\cdots,x_n)\in\mathbb{R}^n\mapsto V(x_1,\cdots,x_n)\in\mathbb{R}$ $(n\in\mathbb{N})$,  $\boldsymbol{\hat x}=(\hat x_1,\cdots,\hat x_n)\in\mathbb{R}^n$ and $1\leq k\leq n$ we denote by $V^+_{x_k}(\boldsymbol{\hat x})$ the right-hand derivative of $V$ with respect to the variable $x_k$ at $\boldsymbol{\hat x}$ and by $V^-_{x_k}(\boldsymbol{\hat x})$ the left-hand derivative of $V$ with respect to the variable $x_k$ at $\boldsymbol{\hat x}$.
 \begin{prop}\label{prop derivatibe of pressure without conditions}
     For all $(b_0,q_0)\in \mathbb{R}^2$ we have 
     \begin{align*}
         p^+_{q}(b_0,q_0)=\sup\{-\mu(\phi):\mu\in M_{b_0,q_0}\} \text{ and }
         p^-_{q}(b_0,q_0)=\inf\{-\mu(\phi):\mu\in M_{b_0,q_0}\}.
     \end{align*}
      \end{prop}
      \begin{proof}
We first show that $p^+_{q}(b_0,q_0)=\sup_{\mu\in M_{b_0,q_0}}\{-\mu(\phi)\}$. By the definition of $p$, for all $q_0<q$ and $\mu\in M_{b_0,q_0}$ we have 
\[
\frac{p(b_0,q)-p(b_0,q_0)}{q-q_0}\geq -\mu(\phi), \text{ which yields that } p^+_{q}(b_0,q_0)\geq\sup_{\mu\in M_{b_0,q_0}}\{-\mu(\phi)\}.
\]
For each $q>q_0$ we fix $\mu_q\in M_{b_0,q}$. Since $M(\rational)$ is compact in the weak* topology, there exist a sequence $\{q_n\}_{n\in\mathbb{N}}\subset (q_0,\infty)$ and $\mu\in M(\rational)$ such that $\lim_{n\to \infty} q_n=q_0$ and $\lim_{n\to\infty} \mu_{q_n}=\mu$. Thus, by the upper semi continuity of the entropy map and the continuity of the function $q\mapsto p(b_0,q)$, we obtain 
\begin{align*}
   &h(\mu)-b_0\lambda(\mu)-q_0\mu(\phi)\geq \limsup_{n\to\infty} (h(\mu_{q_n})-b_0\lambda(\mu_{q_n})-q_n\mu_{q_n}(\phi))
   \\&=\limsup_{n\to\infty} p(b_0,q_n)=p(b_0,q_0) \text{ and thus, }
   \mu\in M_{b_0,q_0}.
\end{align*}
Hence, 
$     p^+_{q}(b_0,q_0)=\lim_{n\to\infty} \frac{p(b_0,q_n)-p(b_0,q_0)}{q_n-q_0}\leq-\lim_{n\to\infty}\mu_{q_n}(\phi)=-\mu(\phi).  
$
Therefore, $p^+_{q}(b_0,q_0)\leq\sup_{\mu\in M_{b_0,q_0}}\{-\mu(\phi)\}$ and thus $p^+_{q}(b_0,q_0)=\sup_{\mu\in M_{b_0,q_0}}\{-\mu(\phi)\}$. By a similar argument, one can show that  $p^-_{q}(b_0,q_0)=\inf_{\mu\in M_{b_0,q_0}}\{-\mu(\phi)\}$.
      \end{proof}

\section{Variational Birkhoff spectrum and Thermodynamic formalism}\label{sec Variational Birkhoff spectrum and Thermodynamic formalism}
In this section, we study the variational Birkhoff spectrum using the pressure function $\pressure(b,q)$.
 We fix a parabolic rational map $\rational$ with $\text{deg}(\rational)\geq 2$ such that $\infty\notin J$ and \eqref{eq assumption simple} holds. We also fix a potential $\phi\in \mathcal{H}$.
Recall that 
$\dimension:=\dim_H(J)$. We first recall Bowen's formula (see, for example, \cite[Theorem 31.2.1 (i)]{UrbanskiRoyMunday})
\begin{thm}\label{thm bowen formula}
    We have
    \[
    \dimension
    =\sup\left\{\frac{h(\mu)}{\lambda(\mu)}:\mu\in M(\rational),\  \lambda(\mu)>0\right\}
    =\min\{t\in\mathbb{R}: P(-t\log|\rational'|)= 0\}.
    \]
\end{thm}

\begin{lemma}\label{lemma indu and nonindu limit}
   We have $\dimension=\dim_H(\coding(\indudomain))$. 
\end{lemma}
\begin{proof}
We define the countable set $\mathcal{Z}:=\bigcup_{\omega\in \Omega}\bigcup_{n\in \mathbb{N}\cup\{0\}}\rational^{-n}(\omega)$.
Then we have $\coding(\indudomain)=\bigcup_{h\in H}\coding([h])\setminus\mathcal{Z}$.
By \cite[Theorem 31.3.9]{UrbanskiRoyMunday}, the support of the unique $\dimension$-conformal measure $m_\dimension$ is $J$. In particular, for every $h\in H$ we have $m_\dimension(\coding([h]))>0$. Moreover, by \cite[Theorem 31.5.1]{UrbanskiRoyMunday}, $m_\dimension$ is a positive constant multiple of either the $\dimension$-dimensional Hausdorff measure or the $\dimension$-dimensional packing measure. Furthermore, by \cite[Remark 31.5.7]{UrbanskiRoyMunday}, $\dimension$ coincides with the packing dimension of $J$. Hence, for every $h\in H$ we have $\dim_H(\coding([h]))=\dimension$. Since $\mathcal{Z}$ is a countable set, we obtain $\dimension=\dim_H(\coding(\indudomain))$. 
\end{proof}

For $\omega\in \Omega$ we define
\begin{align}\label{eq def minrate}
\growth:=(\rate_\omega+1)\rate_\omega^{-1} \text{ and }
\minrate:=\min_{\omega\in \Omega}\growth.
\end{align}

\begin{thm}\label{thm maximal measure and liftable problem}
We have $\indupressure(\dimension,0,0)=0$ and $\minrate^{-1}<\dimension$. Moreover, $\indumeasure_{\dimension,0}(\return)<\infty$  if and only if $\dimension\minrate>2$.  
\end{thm}
\begin{proof}
For $e\in \induedge$, $\Proj^*(e)$ can be written as $\Proj^*(e)=\tau(e) h(e)$, where $\tau(e)\in H\inddiferent^{{|\Proj^*(e)|-2}}\cap\MS^*$ and $h(e)\in H$. Here, $\inddiferent^0:=\{\emptyset\}$, where $\emptyset$ denotes the empty word.
By Lemma \ref{lemma distortion} and Lemma \ref{lemma uniformly expanding}, 
the system 
\[
\left\{R^{-|\tau(e)|}_{\tau(e)h(e)}:\overline{B(z_{h(e)},\kappa)}\rightarrow R^{-|\tau(e)|}_{\tau(e)h(e)}\left(\overline{B(z_{h(e)},\kappa)}\right)\right\}_{e\in \induedge}
\]
can be regarded as a conformal graph directed Markov system in the sense of \cite[Definition 19.3.1]{Urbanskinoninvertible}. Moreover, its limit set is $\coding(\indudomain)$. Thus,  by Bowen's formula \cite[Theorem 19.6.4]{Urbanskinoninvertible} and Lemma \ref{lemma indu and nonindu limit}, we have
\begin{align}\label{eq Bowen indu}
    \dimension=\dim_H(\coding(\indudomain))=\inf\{t\geq 0:\indupressure(t,0,0)\leq 0\}.
\end{align}
Let $K\subset (0,\infty)$ be a compact set such that $\dimension\in K$.
    By \eqref{eq growth rate}, 
    for all $b\in K$ we have 
\begin{align}\label{eq proof of thm maximal measure}
\sum_{e\in \induedge}e^{(-b\log|\indumap'|)([e])}
 \asymp\sum_{\omega\in\Omega}\sum_{n\in\mathbb{N}}{n^{-b\growth}} \asymp \sum_{n\in\mathbb{N}} n^{-b\minrate}.     
\end{align}
 Therefore, since 
 $\CMS$ is finitely primitive and $\log|\indumap'|$ is locally H\"older, we have $\lim_{b\to \minrate^{-1}} \indupressure(b,0,0)=\infty$. 
By Lemma \ref{lemma uniformly expanding}, we have $\lim_{b \to \infty}\indupressure(b,0,0)=-\infty$. Hence, by \eqref{eq Bowen indu}, we obtain $\indupressure(\dimension,0,0)=0$ and $\minrate^{-1}<\dimension$.
By \eqref{eq gibbs} and \eqref{eq growth rate}, we have 
\begin{align*}
\indumeasure_{\dimension,0}(\return)
=\sum_{n\in\mathbb{N}}\sum_{e\in\induedge_n}
n\indumeasure_{\dimension,0}([e])
    \asymp
    \sum_{n\in\mathbb{N}}\sum_{e\in \induedge_n}ne^{-\dimension\log|\indumap'|([e])}
    \asymp
    \sum_{n\in\mathbb{N}}{n^{1-\dimension\minrate}}.
\end{align*}
Therefore, we obtain the last statement and the proof is complete.
\end{proof}

We set $\indumeasure_{\dimension}:=\indumeasure_{\dimension,0}$.
If $\dimension\minrate>2$ then we define $\codmeas_{\dimension}$ by \eqref{eq def lift} using $\indumeasure_{\dimension}$.
If $\dimension\minrate>2$, since $\mu_\dimension$ is a unique ergodic invariant Borel probability measure on $J$ which is equivalent to the $\dimension$-conformal measure $m_\dimension$, we have 
$
    \mu_\dimension=\codmeas_\dimension\circ \coding^{-1}.
$
We also define
$M_{\dimension}:=M_{\dimension,0}.
$

\begin{prop}\label{prop equilibrium state Lambda}
    $M_{\dimension}=\text{Conv}(\{\delta_\omega\}_{\omega\in \Omega})$ if  $\dimension\minrate\leq2$ and $M_{\dimension}=\text{Conv}(\{\delta_\omega\}_{\omega\in \Omega}\cup\{\mu_\dimension\})$ if $\dimension\minrate>2$. 
\end{prop}

\begin{proof}
By Theorem \ref{thm bowen formula}, we have $\bigcup_{\omega\in\Omega}\{\delta_{\omega}\}\subset M_{\dimension}$.
Let $\mu\in M_{\dimension}$ be an ergodic measure such that $\mu\notin \text{Conv}(\{\delta_\omega\}_{\omega\in \Omega})$.
By Remark \ref{rem coding pressure}, there exists $\codmeas\in M(\shift)$ such that $\mu=\codmeas\circ \coding^{-1}$.
Since $\mu\notin \text{Conv}(\{\delta_\omega\}_{\omega\in \Omega})$,  Lemma \ref{lemma equivalent condition not liftable} implies that  $\codmeas(\indudomain)>0$. 
Let $\indumeasure:=\codmeas|_{\indudomain}\circ\Proj/\codmeas(\indudomain)$.
Then, by the variational principle (Theorem \ref{thm variational principle induce}), Theorem \ref{thm maximal measure and liftable problem} and \eqref{eq classical Abramov-Kac's formula}, we obtain
$
0=\indupressure(\dimension,0,0)\geq h(\indumeasure)-\dimension\lambda(\indumeasure)\geq{\indumeasure(\return)}(h(\mu)-\dimension\lambda(\mu))=0,
$
Therefore, $\indupressure(\dimension,0,0)=h(\indumeasure)-\dimension\lambda(\indumeasure)$ and $\indumeasure$ is an equilibrium measure for 
$-\dimension\log |\indumap'|$. By the uniqueness of the equilibrium measure for 
$-\dimension\log |\indumap'|$ (see Theorem \ref{thm equilibrium state induce}), we obtain 
\begin{align}\label{eq proof equilibrium state geometric}
    \indumeasure=\indumeasure_{\dimension}.
\end{align}
We first assume that $\dimension \minrate\leq 2$. Then, by Theorem \ref{thm maximal measure and liftable problem} and \eqref{eq classical Abramov-Kac's formula}, we have $\infty=\indumeasure_{\dimension}(\return)=1/\codmeas(\indudomain)<\infty$. This is a contradiction. Therefore, the set of ergodic measures in $M_\dimension$ is $\bigcup_{\omega\in\Omega}\{\delta_{\omega}\}$. By the ergodic decomposition theorem (see \cite[Theorem 5.1.3]{viana}), $M_{\dimension}=\text{Conv}(\{\delta_\omega\}_{\omega\in \Omega})$.
Next, we consider the case $\dimension\minrate>2$. In this case, by \eqref{eq proof equilibrium state geometric}, the set of ergodic measures in $M_\dimension$ is $\bigcup_{\omega\in\Omega}\{\delta_{\omega}\}\cup \{\mu_\dimension\}$. Thus, by ergodic decomposition theorem, we obtain $M_{\dimension}=\text{Conv}(\{\delta_\omega\}_{\omega\in \Omega}\cup\{\mu_\dimension\})$. 
\end{proof}

By Theorem \ref{thm bowen formula}, for all $b\in (-\infty,\dimension)$ we have $p(b,0)>0$ and thus $(b,0)\in \mathcal{N}$. 
For $b\in(-\infty,\dimension)$ we set $\indumeasure_{b}:=\indumeasure_{b,0}$, $\codmeas_b=\codmeas_{b,0}$ and $\mu_{b}:=\mu_{b,0}$. We recall that $\prange$ is defined by \eqref{eq def flat}.

\begin{lemma}\label{lemma equilibrium measure near Delta}
There exists $\{b_n\}_{n\in\mathbb{N}}\subset (-\infty,\dimension)$ such that $\lim_{n\to\infty}b_n=\dimension$ and   $\lim_{n\to\infty}\mu_{b_n}(\phi)\in \prange$.
    \end{lemma}

\begin{proof}
    We first note that, since $J$ is compact, $M(\rational)$ is also compact with respect to the weak* topology (see \cite[Theorem 6.10]{walters2000introduction}). Let $\{b_n\}_{n\in\mathbb{N}}\subset(-\infty,\dimension)$ be a sequence such that $\lim_{n\to\infty}b_n=\dimension$. Since $M(\rational)$ is compact, there exist a subsequence $\{b_{n_k}\}_{k\in\mathbb{N}}\subset \{b_n\}_{n\in\mathbb{N}}$ and $\mu\in M(\rational)$ such that $\mu_{b_{n_{k}}}$ converges to $\mu$ as $k\to\infty$ in the weak* topology. We shall show that $\mu\in M_{\dimension}$. Since the entropy map is upper semi-continuous, $\mu_{b_{n_{k}}}$ ($k\in\mathbb{N}$) is the equilibrium measure for $-b_{n_{k}}\log |\rational'|$ and the map $b\mapsto p(b,0)$ is continuous on $\mathbb{R}$, we have 
    \[
h(\mu)-\dimension\lambda(\mu)\geq \limsup_{k\to\infty}( h(\mu_{b_{n_k}})-b_{n_k}\lambda(\mu_{b_{n_k}}))=\limsup_{k\to\infty}p(b_{n_k},0)=p(\dimension,0).
    \]
Thus, $\mu\in M_\dimension$. By Proposition \ref{prop equilibrium state Lambda}, we obtain $\lim_{k\to\infty}\mu_{b_{n_k}}(\phi)=\mu(\phi)\in \prange$. 
\end{proof}

\begin{prop}\label{prop non empty B and variational B}
Let $\alpha\in \mathbb{R}$.
The set $B(\alpha)$ is not empty if and only if $\alpha\in[\alpha_{\min},\alpha_{\max}]$. Moreover, if $\alpha\in[\alpha_{\min},\alpha_{\max}]\setminus [\pinf,\psup]$ then there exists $\mu\in M(\rational)$ such that $\lambda(\mu)>0$ and $\mu(\phi)=\alpha$.    
\end{prop}

\begin{proof}
    We assume that $B(\alpha)$ is not empty. Then there exists $\tau\in \MS$ such that 
    $\lim_{n\to\infty}\frac{1}{n}S_n(\phi^*)(\tau)=\alpha.$    
    For a contradiction, we assume that $\alpha\notin[\alpha_{\min},\alpha_{\max}]$. Then there exists $\epsilon>0$ such that $[\alpha-\epsilon,\alpha+\epsilon]\cap[\alpha_{\min},\alpha_{\max}]=\emptyset$. By \cite[Corollary 24.3.5]{UrbanskiRoyMunday}, $\rational:J\rightarrow J$ is topologically exact. In particular, 
    \begin{align}\label{eq primitive MS}
        \text{$\MS$ is primitive,}
    \end{align} 
    that is, there exist $N\in\mathbb{N}$ and a set $\Lambda\subset \MS^N$ such that for all $e,e'\in \edge $ there is $a=a(e,e')\in \Lambda$ for which $e ae'\in \MS^*$. 
Since $\phi^*$ is continuous on $\MS$, there exists $L\in\mathbb{N}$ such that for each $\tilde \tau\in \MS$ with $\tau_i=\tilde \tau_i$ for all $0\leq i\leq L-1$ we have $|\phi^*(\tau)-\phi^*(\tilde \tau)|<\epsilon/8$.
Since     $\lim_{n\to\infty}\frac{1}{n}S_n(\phi^*)(\tau)=\alpha$, there exists $M\in\mathbb{N}$ such that 
    \[
    \left|
    \frac{1}{M}S_M(\phi^*)(\tau)-\alpha
\right|<\frac{\epsilon}{8} \text{ and }  \frac{(L+N)\max\{\sup_{x\in J}|\phi(x)|,|\alpha|\}}{M}<\frac{\epsilon }{8}.
    \]
    Let $a:=a(\tau_{M+L-1},\tau_{0})$ and let $\mu:=\frac{1}{M+L+N}\sum_{i=0}^{M+L+N-1}\delta_{\rational^i(\coding(\hat \tau))}$, where 
    \[
    \hat\tau:=\tau_0\tau_1\cdots\tau_{M+L-1}a_0\cdots a_{N-1}\tau_0\tau_1\cdots\tau_{M+L-1}a_0\cdots a_{N-1} \cdots.
    \]
    Then $\mu\in M(\rational)$ and $|\mu(\phi)-\alpha|<\epsilon/2$. Hence, $[\alpha-\epsilon,\alpha+\epsilon]\cap[\alpha_{\min},\alpha_{\max}]\neq\emptyset$. This is a contradiction. Hence, $\alpha\in [\alpha_{\min},\alpha_{\max}]$.

    Conversely, let $\alpha\in [\alpha_{\min},\alpha_{\max}]$. Since $M(\rational)$ is a convex set, there exists $\mu\in M(\rational)$ such that $\mu(\phi)=\alpha$. By \cite[Corollary 21.15]{DenkerGrillendergerSigmund}, there exists $x\in J$ such that $\lim_{n\to\infty}\frac{1}{n}S_n(\phi)(x)=\alpha$. This implies that $B(\alpha)\neq \emptyset$.

    We will show that the last statement. Let $\alpha\in[\alpha_{\min},\alpha_{\max}]\setminus [\pinf,\psup]$. Then there exists $\mu\in M(\rational)$ such that $\mu(\phi)=\alpha$. Since $\alpha\notin [\pinf,\psup]$, we have $\mu\notin \text{Conv}(\{\delta_\omega\}_{\omega\in \Omega})$. Thus, by Lemma \ref{lemma positive Lyapunov}, we obtain $\lambda(\mu)>0$.
    \end{proof}

For $\alpha\in (\alpha_{\min},\alpha_{\max})$ we define the function $\palpha:\mathbb{R}^2\rightarrow\mathbb{R}$ by 
\[
\palpha(b,q):=p(b,q)+q\alpha=P(q(-\phi+\alpha)-b\log|\rational'|).
\]
We denote by $\minomega$ a rationally indifferent periodic point $\omega\in \Omega$ satisfying $\alpha_{\minomega}=\pinf$ and by $\maxomega$ a rationally indifferent periodic point $\omega\in \Omega$ satisfying $\alpha_{\maxomega}=\psup$.

We define the variational Birkhoff spectrum $\vspectrum:\mathbb{R}\rightarrow [0,\dimension]$ by 
\[
\vspectrum(\alpha):=\sup\left\{\frac{h(\mu)}{\lambda(\mu)}:\mu\in M(\rational),\ \lambda(\mu)>0,\ \mu(\phi)=\alpha \right\}.
\]

\begin{lemma}\label{lemma positivity p alpha}
    For all $\alpha\in (\alpha_{\text{min}},\alpha_{\text{max}})$ and $q\in \mathbb{R}$ we have $\palpha(\vspectrum(\alpha),q)\geq 0$. Moreover, for all $b\in \mathbb{R}$ we have 
    \begin{align}\label{eq asymptotic behavior of alpha pressure}
        \lim_{|q|\to\infty}\palpha(b,q)=\infty.
    \end{align}
\end{lemma}
\begin{proof}
Let $\alpha\in (\alpha_{\min},\alpha_{\max})$ and let $q\in \mathbb{R}$. 
    If $\alpha\in [\alpha_{\minomega},\alpha_{\maxomega}]$ then there exists $p\in[0,1]$ such that $\alpha=p\alpha_{\minomega}+(1-p)\alpha_{\maxomega}$. Then, by using the measure $p\delta_{\minomega}+(1-p)\delta_{\maxomega}$, we can verify that $\palpha(\vspectrum(\alpha),q)\geq 0$. If $\alpha\in (\alpha_{\min},\alpha_{\max})\setminus [\alpha_{\minomega}, \alpha_{\maxomega}]$ then by Proposition \ref{prop non empty B and variational B}, there exists $\{\mu_n\}_{n\in\mathbb{N}}\subset M(\rational)$ such that for all $n\in\mathbb{N}$ we have $\lambda(\mu_n)>0$, $\mu_n(\phi)=\alpha$ and $\lim_{n\to\infty}h(\mu_n)/\lambda(\mu_n)=\vspectrum(\alpha)$. Thus, 
    $
    \palpha(\vspectrum(\alpha),q)\geq h(\mu_n)-\vspectrum(\alpha)\lambda(\mu_n)$ and  letting $n\to\infty$, we obtain $\palpha(\vspectrum(\alpha),q)\ge 0$.  

    Next, we show \eqref{eq asymptotic behavior of alpha pressure}. Let $\alpha\in (\alpha_{\min},\alpha_{\max})$ and let $b\in\mathbb{R}$. Then there exist $\underline{\nu},\overline{\nu}\in M(\rational)$ such that $\underline{\nu}(\phi)<\alpha<\overline{\nu}(\phi)$. Hence, we obtain 
    $
        \lim_{q\to-\infty}\palpha(b,q)
        \geq 
         h(\overline{\nu})+\lim_{q\to-\infty}q(-\overline{\nu}(\phi)+\alpha)-b\lambda(\overline{\nu})=\infty
        \text{ and }
        \lim_{q\to\infty}\palpha(b,q)
        \geq 
         h(\underline{\nu})+\lim_{q\to\infty}q(-\underline{\nu}
         (\phi)+\alpha)-b\lambda(\underline{\nu})=\infty.
    $
\end{proof}

\begin{prop}\label{prop only frat}
     For all $\alpha\in (\alpha_{\text{min}},\alpha_{\text{max}})\setminus \prange$ we have $\vspectrum(\alpha)<\dimension$.
\end{prop}

\begin{proof}
    We first consider the case $\alpha\in (\alpha_{\text{min}},\min \prange)$. For a contradiction, we assume that there exists  $\alpha\in (\alpha_{\text{min}},\min \prange)$ such that $\vspectrum(\alpha)=\dimension$. By Theorem \ref{thm bowen formula} and Lemma \ref{lemma positivity p alpha}, we have $\palpha(\dimension,0)=0$ and $\palpha(\dimension,q)\geq 0$ for all $q\in(0,\infty)$. By the convexity of the function $q\mapsto\palpha(\dimension,q)$, we obtain $(\palpha)_q^+(\dimension,0)\geq0$. However, by Proposition \ref{prop derivatibe of pressure without conditions} and Proposition \ref{prop equilibrium state Lambda}, we have
 $
 (\palpha)_q^+(\dimension,0)=\sup_{\mu\in M_\dimension}\{-\mu(\phi)\}+\alpha= -\min \prange+\alpha<0.
 $
    This is a contradiction. Therefore, for all $\alpha\in (\alpha_{\text{min}},\min \prange)$ we have $\vspectrum(\alpha)<\dimension$. By a similar argument, one can show that  for all $\alpha\in (\max \prange,\alpha_{\text{max}})$ we have $\vspectrum(\alpha)<\dimension$. 
    \end{proof}

\begin{prop}\label{prop relationship}
    For each $\alpha\in (\alpha_{\text{min}},\min \prange)$ (resp. $\alpha\in (\max\prange,\alpha_{\text{max}})$) there exists a unique number $q(\alpha)\in (0,\infty)$ (resp. $q(\alpha)\in (-\infty,0)$) such that $(\vspectrum(\alpha),q(\alpha))\in \mathcal{N}$ and  
    \begin{align}\label{eq relation in the statement}
\palpha(\vspectrum(\alpha),q(\alpha))=0 \text{ and } \frac{\partial}{\partial q}\palpha(\vspectrum(\alpha),q(\alpha))=0.        
    \end{align}
    Moreover, $\mu_{\vspectrum(\alpha),q(\alpha)}$ is the unique measure $\nu\in M(\rational)$ satisfying $\lambda(\nu)>0$, $\nu(\phi)=\alpha$ and $\vspectrum(\alpha)=h(\nu)/\lambda(\nu)$.
\end{prop}
\begin{proof}
If $\alpha_{\text{min}}=\alpha_{\text{max}}$, there is nothing to prove. Thus, we assume that $\alpha_{\text{min}}<\alpha_{\text{max}}$.
Let $\alpha\in(\alpha_{\text{min}},\min \prange)$. Then, by Theorem \ref{thm bowen formula}, Lemma \ref {lemma positivity p alpha} and Proposition \ref{prop only frat}, we have $p(\vspectrum(\alpha),0)>0$ and $p(\vspectrum(\alpha),q)\geq -q\alpha>-q\min \prange\geq \LB$ for all $q\in (0,\infty)$. Hence, $\{\vspectrum(\alpha)\}\times[0,\infty)\subset \mathcal{N}$. 
On the other hand, by Theorem \ref{thm bowen formula}, for all $b\in (-\infty,\dimension)$ we have $\palpha(b,0)>0$ and thus, $(-\infty,\dimension)\times\{0\}\subset \mathcal{N}$. 
Hence, by Theorem \ref{thm regularity of non induced pressure}, the function $\palpha$ is real-analytic on an open set $\mathcal{O}$ containing $\{\vspectrum(\alpha)\}\times [0,\infty)\cup (-\infty,\dimension)\times \{0\}$. We will show that 
\begin{align}\label{eq proof of relation}
\frac{\partial}{\partial q}\palpha(\vspectrum(\alpha),0)<0.    
\end{align}
For a contradiction, we assume that 
\begin{align}\label{eq proof of relation contradiction}
\frac{\partial}{\partial q}\palpha(\vspectrum(\alpha),0)\geq 0.    
\end{align}
We take a small number $\epsilon>0$ such that  $\alpha<\min \prange-\epsilon$. By Lemma \ref{lemma equilibrium measure near Delta}, there exists $b_0\in(\vspectrum(\alpha),\dimension)$ such that  $\min\prange-\epsilon\leq \mu_{b_0}(\phi)$. Then, by Theorem \ref{thm regularity of non induced pressure}, we have 
\begin{align*}
    \frac{\partial}{\partial q}\palpha(b_0,0)=-\mu_{b_0}(\phi)+\alpha<-\min \prange+\epsilon+\alpha<0.
\end{align*}
Combining this with \eqref{eq proof of relation contradiction} and using the continuity of the function $b \in[\vspectrum(\alpha),b_0]\mapsto \frac{\partial}{\partial q}\palpha(b,0)$, we conclude that there exists $b'\in [\vspectrum(\alpha),b_0)$ such that $\frac{\partial}{\partial q}\palpha(b',0)=0$. This yields that $\mu_{b'}(\phi)=\alpha$. Thus, we obtain 
\[
0<\palpha(b',0)=h(\mu_{b'})-b'\lambda(\mu_{b'})\leq h(\mu_{b'})-\vspectrum(\alpha)\lambda(\mu_{b'})\leq h(\mu_{b'})-\frac{h(\mu_{b'})}{\lambda(\mu_{b'})}\lambda(\mu_{b'})=0.
\]
This is a contradiction and we obtain \eqref{eq proof of relation}. Since we assume that $\alpha_{\text{min}}<\alpha_{\text{max}}$, Theorem \ref{thm regularity of non induced pressure} yields that the function $q\mapsto\palpha(\vspectrum(\alpha),q)$ is strictly convex on $[0,\infty)$. Thus, by \eqref{eq proof of relation} and \eqref{eq asymptotic behavior of alpha pressure}, there exists a unique number $q(\alpha)\in (0,\infty)$ such that $\frac{\partial}{\partial q}\palpha(\vspectrum(\alpha),q(\alpha))=0$. In particular $\mu_{\vspectrum(\alpha),q(\alpha)}(\phi)=\alpha$. Moreover, by  Lemma \ref{lemma positivity p alpha}, we obtain
$
0\leq \palpha(\vspectrum(\alpha),q(\alpha))=h(\mu_{\vspectrum(\alpha),q(\alpha)}) -\vspectrum(\alpha)\lambda(\mu_{\vspectrum(\alpha),q(\alpha)})\leq 0
$
and thus, $\palpha(\vspectrum(\alpha),q(\alpha))=0$ and $\vspectrum(\alpha)=h(\mu_{\vspectrum(\alpha),q(\alpha)})/\lambda(\mu_{\vspectrum(\alpha),q(\alpha)})$. Hence, for all $\alpha\in(\alpha_{\text{min}},\min \prange)$, the proof of the first half is complete. By a similar argument, we can show that for all $\alpha\in (\max \prange, \alpha_{\text{max}})$  there exists a unique number $q(\alpha)\in (-\infty,0)$ such that $(\vspectrum(\alpha),q(\alpha))\in \mathcal{N}$ and \eqref{eq relation in the statement} holds. Moreover, we have $\vspectrum(\alpha)=h(\mu_{\vspectrum(\alpha),q(\alpha)})/\lambda(\mu_{\vspectrum(\alpha),q(\alpha)})$. 

Next, we shall show that the second half. Let $\alpha\in (\alpha_{\text{min}},\alpha_{\text{max}})\setminus \prange$ and let $\nu$ be in $M(\rational)$ such that $\lambda(\nu)>0$, $\nu(\phi)=\alpha$ and $\vspectrum(\alpha)=h(\nu)/\lambda(\nu)$. Then, we have
\[
h(\nu)+q(\alpha)(-\nu(\phi)+\alpha)-\vspectrum(\alpha)\lambda(\nu)=0=p(\vspectrum(\alpha),q(\alpha))+q(\alpha)\alpha.
\]
Thus, $\nu$ is an equilibrium measure for $-q(\alpha)\phi-\vspectrum(\alpha)\log|\rational'|$. Therefore, since $(\vspectrum(\alpha),q(\alpha))\in \mathcal{N}$, the uniqueness of an equilibrium measure for the potential $-q(\alpha)\phi-\vspectrum(\alpha)\log|\rational'|$ (Theorem \ref{thm uniquness and existence of the equilibrium state}) yields that $\nu=\mu_{\vspectrum(\alpha),q(\alpha)}$. 
\end{proof}

\begin{prop}\label{prop real analytic}
    The functions $\alpha\mapsto \vspectrum(\alpha)$ and $\alpha\mapsto q(\alpha)$ are real-analytic on $(\alpha_{\text{min}},\alpha_{\text{max}})\setminus \prange$.
\end{prop}

\begin{proof}
If $\alpha_{\text{min}}=\alpha_{\text{max}}$, there is nothing to prove. Thus, we assume that $\alpha_{\text{min}}<\alpha_{\text{max}}$.
We define $G:\mathbb{R}\times\mathcal{N}\rightarrow \mathbb{R}^2$ by $G(\alpha,b,q):= (\palpha(b,q), \frac{\partial}{\partial q}\palpha(b,q))=(p(b,q)+q\alpha, \frac{\partial}{\partial q}p(b,q)+\alpha)$. By Proposition \ref{prop relationship}, for all $\alpha\in (\alpha_{\text{min}},\alpha_{\text{max}})\setminus \prange$ we have $G(\alpha,\vspectrum(\alpha),q(\alpha))=0$. 
We want to apply the implicit function theorem. To do this, it is sufficient to show that 
\begin{align*}
    \text{det}\begin{pmatrix}
   \frac{\partial}{\partial b}\palpha(\vspectrum(\alpha),q(\alpha)) & \frac{\partial}{\partial q}\palpha(\vspectrum(\alpha),q(\alpha)) \\
   \frac{\partial^2}{\partial b\partial q}\palpha(\vspectrum(\alpha),q(\alpha)) & 
   \frac{\partial^2}{\partial q^2}\palpha(\vspectrum(\alpha),q(\alpha))
\end{pmatrix}
\neq 0.
\end{align*}
Note that, by Proposition \ref{prop relationship} and Theorem \ref{thm regularity of non induced pressure}, we have $\frac{\partial}{\partial q}\palpha(\vspectrum(\alpha),q(\alpha))=0$ and $\frac{\partial}{\partial b}\palpha(\vspectrum(\alpha),q(\alpha))=
-\lambda(\mu_{\vspectrum(\alpha),q(\alpha)})
<0$.
Since $\alpha_{\text{min}}<\alpha_{\text{max}}$, Theorem \ref{thm regularity of non induced pressure} implies that $\frac{\partial^2}{\partial q^2}\palpha(\vspectrum(\alpha),q(\alpha))\neq 0$. Therefore, by the implicit function theorem and Theorem \ref{thm regularity of non induced pressure}, we are done. 
\end{proof}

\begin{prop}\label{prop monotone}
    The function $\vspectrum$ is strictly increasing on $(\alpha_{\text{min}},\min \prange)$ and  strictly decreasing on $(\max \prange, \alpha_{\text{max}})$.
\end{prop}

\begin{proof}
    Let $\alpha_1,\alpha_2\in (\alpha_{\text{min}},\min \prange)$.
    We assume that $\alpha_1<\alpha_2$. By Proposition \ref{prop only frat}, we have $b(\alpha_1)<\dimension$. 
    We take a small $\epsilon>0$ with $\alpha_2<\min\prange-\epsilon$ and $b(\alpha_1)<\dimension-\epsilon$. We first show that there exists $\mu\in M(\rational)$ such that  
    \begin{align}\label{eq proof of monotonicity saturated}
        \lambda(\mu)>0,\ \min\prange-\epsilon<\mu(\phi) \text{ and } \dimension-\epsilon<\frac{h(\mu)}{\lambda(\mu)}.
    \end{align}
    By Theorem \ref{thm bowen formula}, there exists $\nu\in M(\rational)$ such that $\lambda(\nu)>0$ and $\dimension-\epsilon<h(\nu)/\lambda(\nu)$. Moreover, there exists $p\in [0,1)$ such that $\min \prange-\epsilon<p\alpha_{\minomega}+(1-p)\nu(\phi)$. We set $\mu:=p\delta_{\minomega}+(1-p)\nu$. Then, by the affinity of the entropy map (\cite[Theorem 8.1]{walters2000introduction}), we obtain $h(\mu)/\lambda(\mu)=h(\nu)/\lambda(\nu)$ and thus, $\mu$ satisfies \eqref{eq proof of monotonicity saturated}.
    
    Since $\alpha_1<\alpha_2<\mu(\phi)$, there exists $\bar p\in (0,1)$ such that $\alpha_2=\bar p\alpha_1+(1-\bar p)\mu(\phi)$.  
We set $\bar\nu=\bar p\mu_{b(\alpha_1),q(\alpha_1)}+(1-\bar p)\mu$. By Proposition \ref{prop relationship}, \eqref{eq proof of monotonicity saturated} and the affinity of the entropy map, we obtain $\bar\nu(\phi)=\alpha_2$ and 
\begin{align*}
&h(\bar \nu)=\bar ph(\mu_{b(\alpha_1),q(\alpha_1)})+(1-\bar p)h(\mu)
\\&>b(\alpha_1)(\bar p\lambda(\mu_{b(\alpha_1),q(\alpha_1)})+(1-\bar p)\lambda(\mu))    
=b(\alpha_1)\lambda(\bar \nu).
\end{align*}
Hence,  we obtain $b(\alpha_1)<b(\alpha_2)$. This implies that $\vspectrum$ is strictly increasing on $(\alpha_{\text{min}},\min \prange)$. By the similar argument, we can show that $b$ is strictly decreasing on $(\max \prange, \alpha_{\text{max}})$.
\end{proof}

\section{Conditional variational principle}\label{sec Conditional variational principle}
In this section, we establish a conditional variational principle and complete the proof of Theorem \ref{thm main}.
As in the previous section, we fix a parabolic rational map $\rational$ with $\text{deg}(\rational)\geq 2$ such that $\infty\notin J$ and \eqref{eq assumption simple} holds and a potential $\phi\in \mathcal{H}$.

Recall that by \eqref{eq primitive MS}, $\MS$ is primitive. 
Thus, there exist $\ell\in\mathbb{N}$ and a set $\Lambda\subset\MS^\ell$ such that, for every $e,e'\in\edge$, there exists $a_{e,e'}\in\Lambda$ for which $ea_{e,e'}e'\in\MS^*$. We fix such an $\ell$ and, for each $e,e'\in\edge$, fix an element $a_{e,e'}\in\Lambda$ with $ea_{e,e'}e'\in\MS^*$.
\begin{lemma}\label{lemma long prabolic}
    There exists a constant $L>0$ such that for all $n\in\mathbb{N}$, $\tau\in \inddiferent^n\cap\MS^n$, $x\in \coding([\tau])$ and $y\in\coding([\tau])$ we have 
    \[
    |x-y|\leq Ln \inf\{|(\rational^{n})'(z)|^{-1}: z\in \comp(\tau a_{\tau_{n-1},h}h), h\in H\}.
    \]
\end{lemma}

\begin{proof}
    We will show that there exists a constant $D>0$ such that for all $n\in\mathbb{N}$, $\omega\in \Omega$, $\tau\in \inddiferent^n_\omega\cap\MS^n$ and $x\in \coding([\tau])$ we have 
\begin{align}\label{eq long parabolic}
    |x-\omega|\leq Dn^{-1/p_\omega}.
\end{align}
By Proposition \ref{prop local property near inddiferent}, there exists a constant $ D\geq 1$ such that for all $n\in\mathbb{N}$, $\omega\in \Omega$ and $z\in \bigcup_{\parabolic h\in \inddiferent_\omega H \cap\MS^2}\coding([\parabolic h])$ we have
\begin{align}\label{eq long parabolic comp inverse}
    \frac{1}{D} \leq \frac{|\rational^{-n}_\omega(z)-\omega|}{n^{-1/p_\omega}}\leq  D  \text{ and }     \frac{1}{D} \leq \frac{|(\rational^{-n}_\omega)'(z)|}{n^{-(p_\omega+1)/p_\omega}}\leq  D .
\end{align}
 Let $n\in\mathbb{N}$ and let $\omega\in\Omega$. 
We fix $\tau\in \inddiferent^n_\omega\cap\MS^n$ and $x\in \coding([\tau])\setminus\{\omega\}$. Then there exists $\tau'\in [\tau]$ such that $x=\coding(\tau')$.  
Since $x\neq \omega$, \eqref{eq n big near inddiferent} implies that $\{m\in\mathbb{N}:\tau_{m}'\in H\}\neq \emptyset$. Define $k:=\min\{m\in\mathbb{N}:\tau_{m}'\in H\}\geq n+1$. Since $2\kappa<\smallnumber$ (see \eqref{eq def kappa} for the definition of $\kappa$ and (I1) and (I2) in Section \ref{sec Dynamics of parabolic rational maps near rationally indifferent fixed points}  for the definition of $\smallnumber$), we have 
\[
x= \rational^{-k}_{\tau_0\cdots\tau_{k}}(\coding(\tau_k\tau_{k-1}\cdots))=\rational_\omega^{-(k-1)}\circ R^{-1}_{\tau_{k-1}\tau_k}(\coding(\tau_k\tau_{k-1}\cdots)).
\]
Therefore, by \eqref{eq long parabolic comp inverse}, we obtain \eqref{eq long parabolic}.  

By \eqref{eq long parabolic}, for all $n\in\mathbb{N}$, $\omega\in\Omega$, $\tau\in \inddiferent_\omega^n\cap \MS^n$, $x\in \coding([\tau])$ and $y\in\coding([\tau])$ we obtain 
\begin{align}\label{eq long parabolic 2}
    |x-y|\leq |x-\omega|+|y-\omega|\leq 2Dn n^{-\frac{p_\omega+1}{p_\omega}}.
\end{align}
Note that since $\text{Crit}(\rational)\cap J=\emptyset$ and $J$ is compact, we have $\sup_{z\in J}|\rational'(z)|^{-1}<\infty$. Thus, by Lemma \ref{lemma distortion}, \eqref{eq long parabolic comp inverse} and the chain rule, there exists a constant $D_1\geq 1$ such that for all $n\in\mathbb{N}$, $\omega\in\Omega$ and $\tau\in \inddiferent_\omega^n\cap \MS^n$ 
\[
D_1^{-1}n^{-\frac{p_\omega+1}{p_\omega}}\leq
\inf\{|(\rational^{n})'(z)|^{-1}: z\in \comp(\tau a_{\tau_{n-1},h}h), h\in H\} \leq D_1n^{-\frac{p_\omega+1}{p_\omega}}.
\]
Combining this with \eqref{eq long parabolic 2} we obtain the desired result.
\end{proof}

\begin{lemma}\label{lemma diam and derivative}
    There exists a constant $C>0$ such that for every $n\in\mathbb{N}$ and $\tau\in \MS^n$ we have 
    \[
    \text{diam}(\coding([\tau]))\le Cn\inf\{|(\rational^{n})'(z)|^{-1}: z\in \coding([\tau a_{\tau_{n-1},h}h]),h\in H\}.
    \]
\end{lemma}

\begin{proof}
We denote by $C$ the constant in Lemma \ref{lemma distortion}. We fix $n\in\mathbb{N}$ and $\tau\in \MS^n$. If $\{\ell\in\{0,\ldots,n-1\}:\tau_\ell\in H\}=\emptyset$, then, by Lemma \ref{lemma long prabolic}, there is nothing to prove. Thus, we assume that $\{\ell\in\{0,\ldots,n-1\}:\tau_\ell\in H\}\neq\emptyset$ and define
\[
k:=\max\{\ell\in \{0,\cdots,n-1\}:\tau_\ell\in H\}.
\]
By Lemma \ref{lemma distortion}, the mean value theorem and the chain rule, for all $x:=\coding(b)\in \coding([\tau])$ and $y:=\coding(\tilde b)\in \coding([\tau])$ we have 
\begin{align}\label{eq diam 1}
    &|x-y|=
    |\rational^{-k}_{\tau_0\cdots \tau_k}(\coding(\tau_k\cdots\tau_{n-1}b_n\cdots))
    -
    \rational^{-k}_{\tau_0\cdots\tau_k}(\coding(\tau_k\cdots\tau_{n-1}\tilde b_n\cdots))|
    \\&\leq C^2\inf_{z\in \comp(\tau_0\cdots\tau_k)}|(\rational^k)'(z)|^{-1}
    |\coding(\tau_k\cdots\tau_{n-1}b_n\cdots)-\coding(\tau_k\cdots\tau_{n-1}\tilde b_n\cdots)|.\nonumber
\end{align}
Since $\kappa<\Koebe \hat q$ (see \eqref{eq def kappa} for the definition of $\kappa$ and \eqref{eq biholomorphic one iteration} for the definition of $\hat q$), \eqref{eq Koebe} implies that for all $x:=\coding(b)\in \coding([\tau])$ and $y:=\coding(\tilde b)\in \coding([\tau])$ we have 
\begin{align}\label{eq diam 2}
 &|\coding(\tau_k\cdots\tau_{n-1}b_n\cdots)-\coding(\tau_k\cdots\tau_{n-1}\tilde b_n\cdots)|
 \\&\leq C\inf_{z\in \pi([\tau_k])}|R'(z)|^{-1} |\coding(\tau_{k+1}\cdots\tau_{n-1}b_n\cdots)-\coding(\tau_{k+1}\cdots\tau_{n-1}\tilde b_n\cdots)|.\nonumber
\end{align}
By the definition of $k$, we have $\tau_{k+1}\cdots\tau_{n-1}\in \inddiferent^{n-(k+1)}\cap\MS^{n-(k+1)}$. Thus, combining Lemma \ref{lemma long prabolic} with \eqref{eq diam 1} and \eqref{eq diam 2}, we obtain the desired result.
\end{proof}

For a continuous function $\psi^*$ on $\MS$ and $n\in \mathbb{N}$ we define 
\[
D_{\psi^*}(n):=\sup_{\tau\in \MS^n}\sup_{\tau_1,\tau_2\in [\tau]}
\left|
S_n(\psi^*)(\tau_1)
-
S_n(\psi^*)(\tau_2)
\right|.
\]
Since $\MS$ is compact, we have 
\begin{align}\label{eq milde distorsion}
   \lim_{n\to\infty} D_{\psi^*}(n)n^{-1}=0.
\end{align}
For $(b,q)\in \mathbb{R}^2$ we set $g^*_{b,q}:=-q\phi^*-b\log |(\rational^*)'|$ and $D_{b,q}(n):=D_{g^*_{b,q}}(n)$.
Since $\MS$ is primitive by \eqref{eq primitive MS}, it is known that (see \cite[Section 2]{Kessebohmerweakgibbs} and \cite[Remark on P.154]{kessebohmer2004multifractaltobeappdated}) for each $(b,q)\in \mathbb{R}^2$ there exists a weak Gibbs measure $m^*_{b,q}$, that is, there exists a constant $C\geq 1$ such that for all $n\in\mathbb{N}$, $\tau\in \MS^n$ and $\tilde \tau\in [\tau]$ we have 
\begin{align}\label{eq weak gibbs}
    (Ce^{D_{b,q}(n)})^{-1}m^*_{b,q}([\tau])\leq{\exp(S_n(g_{b,q}^*)
    (\tilde \tau)-n\codingpressure(b,q))}\leq Ce^{D_{b,q}(n)}m^*_{b,q}([\tau]).
\end{align}
\begin{prop}\label{prop conditional variational principle}
    For all $\alpha\in (\alpha_{\min},\alpha_{\max})\setminus \prange$ we have $b(\alpha)=\dim_H(\mu_{\vspectrum(\alpha),q(\alpha)})=\vspectrum(\alpha)$.
\end{prop} 
\begin{proof}
    By \eqref{eq volume lemma} and Proposition \ref{prop relationship}, for all $\alpha\in (\alpha_{\min},\alpha_{\max})\setminus \prange$ we have $b(\alpha)\geq\dim_H(\mu_{\vspectrum(\alpha),q(\alpha)})=\vspectrum(\alpha)$. We will show that for all $\alpha\in (\alpha_{\min},\alpha_{\max})\setminus \prange$ we have $b(\alpha)\leq\dim_H(\mu_{\vspectrum(\alpha),q(\alpha)})$. 
    
    Let $\alpha\in (\alpha_{\min},\min \prange)$. By Proposition \ref{prop relationship}, we have $(\vspectrum(\alpha),q(\alpha))\in \mathcal{N}$ and $q(\alpha)>0$. We take $b_0>\dim_H(\mu_{\vspectrum(\alpha),q(\alpha)})=\vspectrum(\alpha)$ satisfying $[\vspectrum(\alpha),b_0]\times\{ q(\alpha)\}\subset \mathcal{N}$. We fix $b\in (\vspectrum(\alpha),b_0]$.
    By Theorem \ref{thm regularity of non induced pressure}, the function $t\mapsto \pressure(t,q(\alpha))$ is strictly decreasing on $[\vspectrum(\alpha),b_0]$. 
    Choose $\epsilon>0$ sufficiently small so that
    \begin{align}\label{eq conditional small epsilon}
    (q(\alpha)+2)\epsilon<\pressure(\vspectrum(\alpha),q(\alpha))-\pressure(b,q(\alpha))=-q(\alpha)\alpha-\pressure(b,q(\alpha))    
    \end{align}
    (see Proposition \ref{prop relationship}).
    For $n\in\mathbb{N}$ we define 
    \[
    B_{n}:=\left\{x\in J: \left|\ell^{-1}S_\ell(\phi)(x)-\alpha\right|<\epsilon \text{ for all }\ell\geq n\right\}.
    \]
     Then we have $B(\alpha)\subset \bigcup_{n\in\mathbb{N}} B_n$. Therefore, by \cite[Theorem 15.3.4 and Theorem 15.3.5]{Urbanskinoninvertible}, we obtain 
     $
     b(\alpha)\leq \sup_{n\in\mathbb{N}} \dim_H(B_n).
     $
     Since $b$ is an arbitrary number in $(\dim_H(\mu_{\vspectrum(\alpha),q(\alpha)}),b_0]$, if we show that $\dim_H(B_n)\leq b$ for all $n\in\mathbb{N}$ then we obtain $b(\alpha)\leq \dim_H(\mu_{\vspectrum(\alpha),q(\alpha)})$. 
     
     We fix $n\in\mathbb{N}$.
     By \eqref{eq milde distorsion}, there exists $N\in\mathbb{N}$ such that for all $\ell\geq N$ we have 
    \begin{align}\label{eq mild distorsion b,q and b,0}
    D_{b,q(\alpha)}(\ell)\leq \epsilon \ell \text{ and } D_{b,0}(\ell)\leq \epsilon \ell.
    \end{align}
     For $\ell\geq N$ we define
     $
     E_\ell:=\{\tau\in \MS^\ell: \coding([\tau])\cap B_n\neq \emptyset\}.
     $
     Then for all $\ell\geq N$ we have  
     \begin{align}\label{eq conditional cover}
         B_n\subset \bigcup_{\tau\in E_\ell} \coding([\tau]) \text{ and }\lim_{\ell\to\infty}\max_{\tau\in E_\ell}\text{diam}(\coding([\tau]))=0
     \end{align}
     by \eqref{eq uniform decay of cylinders}. For $\ell \geq N$ and $\tau\in E_\ell $ we fix $h\in H$, $\coding(\tau_1)\in \coding([\tau a_{\tau_{\ell-1},h}h])$ and $\coding( \tau_2)\in \coding([\tau])\cap B_n$. By Lemma \ref{lemma diam and derivative}, \eqref{eq weak gibbs}, \eqref{eq code and noncode pressure} and  \eqref{eq mild distorsion b,q and b,0}, for all $\ell \geq N$ we obtain
     \begin{align*}
         &\sum_{\tau\in E_\ell} (\text{diam}(\coding([\tau])))^b 
         \leq  C_1\ell\sum_{\tau\in E_\ell}
         e^{
         S_\ell(-b\log |(R^*)'|)(\tau_1)
         }
         \\&\leq C_1\ell e^{D_{b,0}(\ell)}
         \sum_{\tau\in E_\ell}
         {e^{\ell \pressure(b,q(\alpha))+S_\ell(q(\alpha)\phi^*)(\tau_2)}}
         \exp(S_\ell(g_{b,q(\alpha)}^*)(\tau_2)-\ell\pressure(b,q(\alpha)))
         \\&
         \leq C_1C_2\ell e^{D_{b,0}(\ell)}e^{D_{b,q(\alpha)}(\ell)}e^{
         \ell(
         \pressure(b,q(\alpha))+q(\alpha)(\epsilon+\alpha)
         )
         }
         \sum_{\tau\in E_\ell}m_{b,q(\alpha)}^*([\tau])
         \\&\leq C_1C_2\ell\exp(\ell (\pressure(b,q(\alpha))+q(\alpha)\alpha+( q(\alpha)+2)\epsilon) ),
     \end{align*}
     where $C_1$ denotes the constant in Lemma \ref{lemma diam and derivative} and $C_2$ denotes the constant in \eqref{eq weak gibbs}. Note that $C_1$ and $C_2$ are independent of $\ell\in\mathbb{N}$. Thus, by \eqref{eq conditional small epsilon}, we obtain $\lim_{\ell\to\infty}\sum_{\tau\in E_\ell} (\text{diam}(\coding([\tau])))^b=0$. Hence, by \eqref{eq conditional cover}, we obtain $\dim_H(B_n)\leq b$ for all $n\in\mathbb{N}$. This implies that for all $\alpha\in (\alpha_{\min},\min \prange)$ we have $b(\alpha)\leq\dim_H(\mu_{\vspectrum(\alpha),q(\alpha)})$. By a similar argument, we can show that for all $\alpha\in (\max \prange, \alpha_{\max})$ we have $b(\alpha)\leq\dim_H(\mu_{\vspectrum(\alpha),q(\alpha)})$. This completes the proof.    
\end{proof}

\emph{Proof of Theorem \ref{thm main}}
(M1) follows from Propositions \ref{prop relationship} and \ref{prop conditional variational principle}. (M2) follows from Propositions \ref{prop real analytic} and  \ref{prop conditional variational principle}. (M3) follows from Propositions \ref{prop only frat}, \ref{prop monotone} and \ref{prop conditional variational principle}.
\qed

\section{Proof of Theorem \ref{thm parabolic ragion}}\label{sec proof Theorem parabolic region}
In this section, we show Theorem \ref{thm parabolic ragion}.
As in the previous section, we fix a parabolic rational map $\rational$ with $\text{deg}(\rational)\geq 2$ such that $\infty\notin J$ and \eqref{eq assumption simple} holds.

We begin with the following version of \cite[Lemma 2.3]{takahasi}. For $F\subset \MS^*$ we define $[F]:=\bigcup_{\tau\in F}[\tau]$. 
For $\ell\in\mathbb{N}$ and $a\in \edge$ we denote by $\MS^\ell(a)$ the set of $\tau\in \MS^\ell$ for which $\tau_0=\tau_{\ell-1}=a$.
\begin{lemma}\label{lemma cylinder approximation}
    Let $m\in \mathbb{N}$ and let $a\in \edge$. Let $\psi^*_1,\cdots, \psi_m^*:\MS\rightarrow \mathbb{R}$ is continuous and let $\codmeas\in M(\shift)$ be ergodic with $\codmeas([a])>0$. Then for any $\epsilon>0$ there exists $N\in\mathbb{N}$ such that for each $n\geq N$ there exist $\ell\geq n$ and $F^\ell\subset \MS^{\ell}(a)$ such that for all $1\leq j\leq m$ we have 
    \[
    \left|\frac{1}{\ell}\log \# F^\ell-h(\codmeas)\right|<\epsilon \text{ and }\sup_{\tau\in [F^\ell]}\left|\frac{1}{\ell}S_\ell(\psi_j^*)-\codmeas(\psi_j^*)\right|<\epsilon.
    \]
\end{lemma}
\begin{proof}
    Let $0<1/s<\codmeas([a])$ and let $0<\epsilon<1$. By \eqref{eq milde distorsion}, there exists $N\in\mathbb{N}$ such that 
    \begin{align}\label{eq cylinder small error}
       \max_{1\leq j\leq m} \sup_{n \geq N} \frac{1}{n} D_{\psi_j^*}(n)<\frac{\epsilon}{2} \text{ and } \sup_{n\geq N}\frac{|\log 2s \epsilon n|+1}{n}<\frac{\epsilon}{3}.
    \end{align}
    For $q\in \mathbb{N}$
    we denote by $B^q$ the set of $\tau\in \tau\in \{\MS^q:\tau_0=a\}$ for which
    \begin{align}\label{eq cylinder entropy and mu cylinder}
        \exp(-(h(\mu^*)+\epsilon/3)q)<\codmeas([\omega])<\exp(-(h(\mu^*)-\epsilon/3)q)
    \end{align}
    and there exists $\tilde \tau\in [\tau]$ such that 
    \begin{align}\label{eq cylinder average}
        \max_{1\leq j\leq m} \sup_{l \geq q} \left|\frac{1}{l}S_l(\psi_j^*)(\tilde \tau)-\codmeas(\psi_j^*)\right|<\frac{\epsilon}{2}, 
    \end{align}
    \begin{align}\label{eq cylinder how many a}
        \frac{1}{q}S_q 1_{[a]}(\tilde \tau)<\left(1+\frac{\epsilon}{3}\right)\codmeas([a]) \text{ and }\frac{1}{q}S_{\lfloor(1+\epsilon)q\rfloor} 1_{[a]}(\tilde \tau)>\left(1+\frac{\epsilon}{2}\right)\codmeas([a]). 
    \end{align}
    Here, $\lfloor\cdot\rfloor$ denotes the floor function and $1_{A}$ denotes the indicator function of $A\subset \MS$. For each $q\in \mathbb{N}$ and $l\geq q$ we define $B^{q,l}:=\{\tau\in B^q:\inf \{j\geq q:[a]\cap \sigma^j([\tau])\neq \emptyset\}=l\}$. Note that by \eqref{eq cylinder how many a}, for each $q\in\mathbb{N}$ we have $B^q=\bigcup_{l=q}^{\lfloor(1+\epsilon)q\rfloor}B^{q,l}$. 
    Since the partition $\{[e]:e\in \edge\}$ is a generator and $\#\edge<\infty$, Kolmogorov-Sinai theorem (see, for example, \cite[Theorem 4.18]{walters2000introduction}) implies that $h(\mu)$ is equal to the entropy of $\shift$ with respect to the partition $\{[e]:e\in \edge\}$. 
    Thus, by Shannon-McMillan-Breiman's theorem (see, for example, \cite[Theorem 2.5.5]{PrzytyckiUrbanski}) and Birkhoff's ergodic theorem, we obtain 
    $
    \lim_{q\to\infty} \codmeas([B^q])=\codmeas([a]).
$
    Let $n\geq N$ and let $q\geq n$ be such that  $\codmeas([B^q])>\codmeas([a])/2$. Then, since $B^q=\bigcup_{l=q}^{\lfloor(1+\epsilon)q\rfloor}B^{q,l}$ and $1/s<\codmeas([a])$, there exists $\ell \in \{q,\cdots, \lfloor(1+\epsilon)q\rfloor\}$ such that $\codmeas([B^{q, \ell}])\geq (2s\epsilon q)^{-1}$. By \eqref{eq cylinder entropy and mu cylinder}, we obtain 
    \begin{align}\label{eq cylinder count}
        ({2s\epsilon q})^{-1} \exp((h(\mu^*)-\epsilon/3)q)\leq \# B^{q,\ell} \leq \exp((h(\mu^*)+\epsilon/3)q).
    \end{align}
    For each $\tau \in B^{q,\ell}$ we fix $\tau'=\tau_0'\tau_1'\cdots\in [a]$ such that $\tau'\in [\tau]$ and $\tau'_\ell=a$. We set $F^{\ell+1}:=\{\tau'_0\cdots\tau_\ell':\tau\in B^{q,\ell}\}$. Then we have $F^{\ell+1}\subset \MS^{\ell+1}(a)$ and $\# F^{\ell+1}=\#B^{q,\ell}$. Moreover, by \eqref{eq cylinder small error} and \eqref{eq cylinder count}, we obtain 
    \[
    \left|\frac{1}{\ell+1}\log \#F^{\ell+1}-h(\codmeas)\right|\leq\left(1-\frac{q}{\ell}\right)h(\mu)+{\epsilon}<2(h_{\text{top}}(\shift)+1)\epsilon,
    \]
    where $h_{\text{top}}(\shift)$ denotes the topological entropy of the dynamical system $(\MS,\shift)$ (see \cite[Definition 7.6 and Theorem 8.6]{walters2000introduction}). Combining this with \eqref{eq cylinder average} and \eqref{eq cylinder small error}, we obtain the desired result. 
\end{proof}

\begin{lemma}\label{lemma hyperbolic ergodic approximation}
    For a continuous function $\phi$ on $J$ and $\alpha\in \prange$ there exists a sequence $\{\eta_{i_n}^*\}_{n\in\mathbb{N}} \subset 
    M(\shift)$ of ergodic measures such that  
    \begin{align}\label{eq hyperbolic ergodic appro large dimension}
    \lim_{n\to\infty}i_n=\infty,\ 
     \lim_{n\to\infty}\frac{h(\eta_{i_n}^*)}{\lambda(\eta_{i_n}^*)}=\dimension 
    \end{align}
     and for each $n\in\mathbb{N}$ and $h\in H$ we have 
    \begin{align}\label{eq hyperbolic ergodic appro nice appro}
        \eta_{i_n}^*([h])>0,\ h(\eta_{i_n}^*)>2i_n^{-1} \text{ and }|\eta_{i_n}^*(\phi)-\alpha|<i_n^{-1/2}. 
    \end{align}
\end{lemma}

\begin{proof}
Let $\Lambda\subset \CMS^M$ ($M\in\mathbb{N}$) 
be a finite set witnessing the finite primitivity of $\CMS$. 
Let $N:=\max\{\max_{[e]}\return:\text{there exist $1\leq i\leq M$ and $\tau\in \Lambda$ such that }\tau_i=e\}$. Then for each $n\geq N$, $\CMSn:=\CMS\cap (\bigcup_{k=1}^{n}\induedge_{k})^{\mathbb{N}\cup\{0\}}$ is a primitive Markov shift over a finite alphabet. As in \eqref{eq def topological induced pressure}, for each $n\geq N$ we define the topological pressure $\tilde P_n(-\dimension \log |\indumap'|_{\CMSn}|)$ of $-\dimension \log |\indumap'|_{\CMSn}|$ with respect to the Markov shift $(\CMSn, \indushift|_{\CMSn})$. The compact approximation property of the topological pressure (see \cite[Theorem 2.1.5]{mauldin2003graph}) and Theorem \ref{thm maximal measure and liftable problem} imply that 
\[
\lim_{n\to\infty}\tilde P_n(-\dimension \log |\indumap'|_{\CMSn}|)= \tilde P(-\dimension\log |\indumap'|)=0.
\]
For $n\geq N$ we denote by the unique ergodic equilibrium measure $\indumeasure_n\in M(\indushift)$ for $-\dimension \log |\indumap'|_{\CMSn}|$. Note that by Lemma \ref{lemma uniformly expanding}, we have $\inf \{\lambda(\indumeasure):\indumeasure\in M(\indushift)\}>0$. Thus, by the variational principle (Theorem \ref{thm variational principle induce}) and the Gibbs property \eqref{eq gibbs}, for each $n\geq N$ and $h\in H$ we have 
$\indumeasure_n\circ \Proj^{-1}([h])>0\text{ and }
    \lim_{n\to\infty} {h(\indumeasure_n)}/{\lambda(\indumeasure_n)}=\dimension+{P_n(-\dimension \log |\indumap'|_{\CMSn}|)}/{\lambda(\indumeasure_n)}=\dimension. 
$
For each $n\geq N$ we define $\codmeas_{n}$ by \eqref{eq def lift} using $\indumeasure_{n}$. Then, by \eqref{eq Abramov-Kac's formula}, for each $n\geq N$ and $h\in H$ we have 
\begin{align}\label{eq ergodic appro 1}
\codmeas_n([h])>0
\text{ and }
    \lim_{n\to\infty}\frac{h(\codmeas_n)}{\lambda(\codmeas_n)}=    \lim_{n\to\infty}\frac{h(\indumeasure_n)}{\lambda(\indumeasure_n)}=\dimension 
\end{align}
Recall that $\minomega$ is a rationally indifferent periodic point $\omega\in \Omega$ satisfying $\alpha_{\minomega}=\pinf$ and $\maxomega$ is a rationally indifferent periodic point $\omega\in \Omega$ satisfying $\alpha_{\maxomega}=\psup$. 

We first consider the case where $\liminf_{n\to\infty}h(\codmeas_n)=0$ and $\alpha=\alpha_{\omega'}\in \{\alpha_{\minomega},\alpha_{\maxomega}\}$ for some $\omega'\in \{{\minomega},{\maxomega}\}$.
Let $\tau(\omega')\in \MS$ satisfy $\pi(\tau(\omega'))=\omega'$.
By taking a subsequence if necessary, we may assume that $\lim_{n\to\infty}h(\codmeas_n)=0$ and $(|\alpha|+\|\phi\|_\infty+1)^{-2}h(\codmeas_n)<1$ for all $n\geq N$, where $\|\phi\|_\infty:=\sup\{|\phi(x)|:x\in J\}$.
For each $n\geq N$ we set $p_n:=16^{-1}(|\alpha|+\|\phi\|_\infty+1)^{-2}h(\codmeas_n)$
 and  $\nu_n^*:=p_n\codmeas_n+(1-p_n)\delta_{\tau(\omega')}$.
Then for all $n\geq N$ we have 
\begin{align}\label{eq hyperbolic appro clear}
    |\nu_n^*(\phi^*)-\alpha|\leq (|\alpha|+\|\phi\|_\infty)p_n
\end{align}
Note that  by the affinity of the entropy map (\cite[Theorem 8.1]{walters2000introduction}), we have 
\begin{align}\label{eq ergodic appro 3}
    h(\nu_n^*)=p_n h(\codmeas_n) \text{ and }\lambda(\nu_n^*)=p_n\lambda(\codmeas_n)
\end{align}
By \cite[Main Theorem]{takahasi}, for each $n\geq N$ there exists an ergodic measure $m_n^*\in M(\shift)$ such that 
\begin{align}\label{eq ergodic appro 4}
   &|h(m_n^*)-h(\nu_n^*)|<\min\left\{\frac{h(\nu_n^*)}{2},\frac{\lambda(\nu_n^*)}{n}\right\}, |m_n^*(\phi^*)-\nu_n^*(\phi^*)|<\frac{h(\nu_n^*)}{16},
   \\& |\lambda(m_n^*)-\lambda(\nu_n^*)|<\frac{\lambda(\nu_n^*)}{n}
   \text{ and } \min_{h\in H} m_n^*([h])>0
   . \nonumber
\end{align}
Then since $\sup_{n\geq N}|h(m_n^*)/\lambda(m_n^*)|\leq \delta$ by Theorem \ref{thm bowen formula}, we have, as $n\to \infty$, 
\begin{align*}
    \left|\frac{h(m_n^*)}{\lambda(m_n^*)}-\frac{h(\nu_n^*)}{\lambda(\nu_n^*)}\right|=
    \left|
    \frac{(\lambda(\nu_n^*)-\lambda(m_n^*))h(m_n^*)}{\lambda(m_n^*)\lambda(\nu_n^*)}
    +\frac{h(m_n^*)-h(\nu_n^*)}{\lambda(\nu_n^*)}
    \right|\leq\frac{\dimension+1}{n}\to0.
\end{align*}
Combining this with \eqref{eq ergodic appro 1} and \eqref{eq ergodic appro 3} and noting that $p_n>0$ for all $n\geq N$,
\begin{align}\label{eq hyperbolic appro case 1}
   \left| \frac{h(m_n^*)}{\lambda(m_n^*)}-\dimension\right|\leq\left|\frac{h(m_n^*)}{\lambda(m_n^*)}-\frac{h(\nu_n^*)}{\lambda(\nu_n^*)}\right|+\left|\frac{h(\mu_n^*)}{\lambda(\mu_n^*)}-\dimension\right| \to 0 \text{ as }n\to \infty. 
\end{align}
Since $\lim_{n\to\infty} h(\nu_n^*)=\lim_{n\to\infty}p_nh(\codmeas_n)=0$, by enlarging $N$ if necessary, we may assume that $h(\nu_n^*)<1$. 
For each $n\geq N$ we set $i_n:=\lfloor4/h(\nu_n^*)\rfloor+1$ and $\eta_{i_n}^*:=m_{n}^*$.
By \eqref{eq hyperbolic appro clear} and \eqref{eq ergodic appro 4}, for each $n\geq N$ we have 
\begin{align*}
&|\eta_{i_n}^*(\phi^*)-\alpha|\leq|m_n^*(\phi^*)-\nu_n^*(\phi^*)|+|\nu_n^*(\phi^*)-\alpha|<\frac{1}{i_n^{1/2}}
\text{, }
    h(\eta_{i_n}^*)>\frac{h(\nu_n^*)}{2}>\frac{2}{i_n} .
\end{align*}
Combining this with \eqref{eq hyperbolic appro case 1} and \eqref{eq ergodic appro 4}, the proof is complete in this case.

Next we consider the case where $\liminf_{n\to\infty}h(\codmeas_n)>0$ and $\alpha=\alpha_{\omega'}\in \{\alpha_{\minomega},\alpha_{\maxomega}\}$ for some $\omega'\in \{{\minomega},{\maxomega}\}$. 
Let $\tau(\omega')\in \MS$ satisfy $\pi(\tau(\omega'))=\omega'$.
By enlarging $N$ if necessary, we may assume that $c:=\inf\{h(\codmeas_n):n\geq N\}>0$.
For each $n\geq N$ we set $p_n:=3/(nc)$ and $\nu_n^*:=p_n\codmeas_n+(1-p_n)\delta_{\tau(\omega')}$.
By \cite[Main Theorem]{takahasi}, for each $n\geq N$ there exists an ergodic measure $\eta_n^*\in M(\shift)$ such that 
\begin{align}\label{eq ergodic appro case 2}
   &|h(\eta_n^*)-h(\nu_n^*)|<\min\{n^{-1},h(\nu_n^*)/n\},\  |\eta_n^*(\phi^*)-\nu_n^*(\phi^*)|<n^{-1},
   \\& |\lambda(\eta_n^*)-\lambda(\nu_n^*)|<\lambda(\nu_n^*){n}^{-1}
   \text{ and } \min_{h\in H} \eta_n^*([h])>0
   . \nonumber
\end{align}
Then we have $\lim_{n\to\infty}h(\eta^*_n)/\lambda(\eta_n^*)=\dimension$ by \eqref{eq ergodic appro 1}. Moreover, for all $n\geq N$ we have $h(\eta_n^*)>p_nh(\codmeas_n)-1/n\geq2/n$ and $|\eta_n^*(\phi^*)-\alpha|\ll 1/n $.

We then consider the case where $\alpha\in (\alpha_{\minomega},\alpha_{\maxomega})$. In this case, for each $n\geq N$ there exist $p_n\in (0,1]$ and $\omega_n\in \{{\minomega}, {\maxomega}\}$ such that the measure $\nu_n^*:=p_n \codmeas_n+(1-p_n)\delta_{\tau(\omega_n)}$ satisfies $\nu_n^*(\phi)=\alpha$, where $\tau(\omega_n)$ satisfies $\pi(\tau(\omega_n))=\omega_n$.
By \cite[Main Theorem]{takahasi}, for each $n\geq N$ there exists an ergodic measure $m_n^*\in M(\shift)$ such that 
\begin{align}\label{eq ergodic appro 5}
   &|h(m_n^*)-h(\nu_n^*)|<\min\left\{\frac{1}{4n(h(\nu_n^*))^{-1}+1},\frac{\lambda(\nu_n^*)}{n}\right\},
   \\&|m_n^*(\phi^*)-\nu_n^*(\phi^*)|<\frac{h(\nu_n^*)}{4n},
    |\lambda(m_n^*)-\lambda(\nu_n^*)|<\frac{\lambda(\nu_n^*)}{n}
   \text{ and } \min_{h\in H} m_n^*([h])>0
   . \nonumber
\end{align}
Then as in the case where $\liminf_{n\to\infty} h(\codmeas_n)=0$ and $\alpha\in \{\alpha_{\minomega},\alpha_{\maxomega}\}$, we can show $\lim_{n\to\infty}h(m_n^*)/\lambda(m_n^*)=\dimension$. For each $n\geq N$ we set $i_n:=\lfloor4n/h(\nu_n^*)\rfloor+1$ and $\eta_{i_n}^*:=m_{n}^*$.
By enlarging $N$ if necessary, we may assume that $i_n-1\geq 1$ for all $n\geq N$.
Then for each $n\geq N$ we have 
\begin{align*}
|\eta_{i_n}^*(\phi^*)-\alpha|<(i_n-1)^{-1},\ 
    h(\eta_{i_n}^*)=h(m_{n}^*)>{(1-(4n)^{-1})h(\nu_n^*)}>{2}{i_n}^{-1}.
\end{align*}

The remaining case is when $\mu_\dimension$ is finite and $\alpha\in F\setminus[\alpha_{\minomega},\alpha_{\maxomega}]$.  We assume that $\mu_\dimension$ is finite, or equivalently, $\dimension\minrate>2$ (see Theorem \ref{thm maximal measure and liftable problem}). Let $\alpha\in F\setminus[\alpha_{\minomega},\alpha_{\maxomega}]$. By \eqref{eq Abramov-Kac's formula} and \eqref{eq gibbs}, $h(\codmeas_{\dimension})/\lambda(\codmeas_{\dimension})=\dimension$ and $\codmeas_{\dimension}([h])>0$ for all $h\in H$. 
Since $\alpha\in F\setminus[\alpha_{\minomega},\alpha_{\maxomega}]$ there exist $p\in (0,1]$ and $\omega\in \{{\minomega}, {\maxomega}\}$ such that the measure $\nu^*:=p \codmeas_\dimension+(1-p)\delta_{\tau(\omega)}$ satisfies $\nu^*(\phi)=\alpha$, where $\tau(\omega)$ satisfies $\pi(\tau(\omega))=\omega$. By a similar argument, in the case where $\liminf_{n\to\infty}h(\codmeas_n)>0$ and $\alpha\in \{\alpha_{\minomega},\alpha_{\maxomega}\}$, we can construct a sequence $\{\eta_{i_n}^*\}_{n\in\mathbb{N}} \subset 
    M(\shift)$ of ergodic measures satisfying \eqref{eq hyperbolic ergodic appro large dimension} and \eqref{eq hyperbolic ergodic appro nice appro}.
\end{proof}

\begin{lemma}\label{lemma boundary}
    For  $k\in\mathbb{N}\cup\{0\}$ and $\tau,\tau'\in \MS$ such that $\coding(\tau)=\coding(\tau')$ and $\tau_k\neq \tau'_k$ we have $\tau_\ell\neq \tau_\ell'$ for all $\ell\geq k$. 
\end{lemma}
\begin{proof}
    For a contradiction we assume that $\{\ell>k:\tau_\ell= \tau_\ell'\}\neq \emptyset$. we set $\ell:=\min\{\ell>k:\tau_\ell= \tau_\ell'\}$, $x:=R^k(\coding(\tau))=R^k(\coding(\tau'))$ and $m:=\ell-k$. 
     Recall that $\rational^m:A_{\tau_k\cdots\tau_\ell}\rightarrow A_{\tau_\ell}$ and $\rational^m:A_{\tau_k'\cdots\tau_\ell'}\rightarrow A_{\tau_\ell}$ are bijection. We denote by $R^{-m}_{\tau_k}$ (resp. $R^{-m}_{\tau_k'}$) the inverse of $\rational^m:A_{\tau_k\cdots\tau_\ell}\rightarrow A_{\tau_\ell}$ (resp. $\rational^m:A_{\tau_k'\cdots\tau_\ell'}\rightarrow A_{\tau_\ell}$).
     Since $J\cap \text{Crit}(\rational^m)=\emptyset$, there exists $\epsilon>0$ such that $\rational^m:U\rightarrow B(\rational^m(x),\epsilon)$ is biholomorphic, where $U$ denotes the connected component of $\rational^{-m}(B(\rational^m(x),\epsilon))$ containing $x$. We denote by $R^{-m}_{U}$ its inverse. 
    Since $\rational^m(x)\in A_{\tau_\ell}=\overline{\text{Int}(A_{\tau_\ell})}$ and $x\in U$, there exists $y\in \text{Int}(A_{\tau_\ell})\cap B(\rational^m(x),\epsilon)$ such that $\rational^{-m}_{\tau_k}(y),\rational^{-m}_{\tau_k'}(y)\in U$.
    Then $\rational^{-m}_{\tau_k}(y)=\rational^{-m}_U(y)=\rational^{-m}_{\tau_k'}(y)$.
    Since $\rational(\bigcup_{e\in \edge}\partial A_e)\subset \bigcup_{e\in \edge}\partial A_e$ and $y\in \text{Int}(A_{\tau_\ell})$, we obtain
     $
         \rational^{-m}_{\tau_k}(y)\in \text{Int}(A_{\tau_k}) \text{ and }\rational^{-m}_{\tau_k'}(y)\in \text{Int}(A_{\tau_k'})
     $. Hence, $\rational^{-m}_{\tau_k}(y)=\rational^{-m}_{\tau_k'}(y)\in\text{Int}(A_{\tau_k})\cap \text{Int}(A_{\tau_k'})$. Since $\tau_k\neq \tau_k'$, this is a contradiction.
\end{proof}

\emph{Proof of Theorem \ref{thm parabolic ragion}.}
    Let $\alpha\in \prange$. We take a sequence $\{\eta_{i_n}^*\}_{n\in\mathbb{N}}\subset M(\shift)$ of ergodic measures satisfying \eqref{eq hyperbolic ergodic appro large dimension} and \eqref{eq hyperbolic ergodic appro nice appro}. In addition, we may assume that $\{i_n\}_{n\in\mathbb{N}}$ is strictly increasing. 
    For simplicity, we write $\{\eta_{i_n}^*\}_{n\in\mathbb{N}}$ as $\{\eta_{i}^*\}_{i\in\mathbb{N}}$. 
    We fix $h\in H$ and $a\in \MS^*$ such that $hah\in \MS^*$. We set $c:=\sup_{\tau\in [a]}S_{|a|}(|\phi^*|)(\tau)$.
    By Lemma \ref{lemma cylinder approximation}, for any $i\in\mathbb{N}$ there exist $\ell_i\in \mathbb{N}$ and $F^{\ell_i}\subset \MS^{\ell_i}(h)$ such that 
    \begin{align}\label{eq non intresting ine}
     \max\left\{\sup_{[a]}\log|(\rational^{|a|})'\circ\coding|,\ ic+i|a||\alpha|, \ \frac{1}{\lambda(\eta^*_{i})+2i^{-1}} \right\}\leq \frac{\ell_i}{i^2},
    \end{align}
    \begin{align}\label{eq nice cylinder appro}
        &\left|\frac{1}{\ell_i}\log \#F^{\ell_i}-h(\eta_i^*)\right|\leq \frac{1}{i^2},\ \sup_{\tau\in [F^{\ell_i}]}\left|\frac{1}{\ell_i}S_{\ell_i}(\phi^*)-\eta^*_i(\phi^*)\right|\leq \frac{1}{i^{1/2}}
        \text{ and }
        \\&\sup_{\tau\in [F^{\ell_i}]}\left|\frac{1}{\ell_i}\log |(\rational^{\ell_i})'\circ \coding|-\lambda(\eta_i^*)\right|\leq \frac{1}{i^2}.\nonumber
    \end{align}
    We inductively construct a strictly increasing sequence $\{t_k\}_{k\in\mathbb{N}}\subset \mathbb{N}$ satisfying the following conditions: for all $i\geq 2$ we have
    \begin{align}\label{eq flat to define r_t}
        t_{i-1}\ell_{i-1}\left(\lambda(\eta_{i-1}^*)+\frac{3}{(i-1)^2}\right)<\ell_{i}\left(\lambda(\eta_{i}^*)+\frac{t_i+2}{i^2}\right),
    \end{align}
\begin{align}\label{eq to bound derivative}
    \sum_{j=1}^{i-1}t_j\ell_j\left(\lambda(\eta_j^*)+\frac{2}{j^2}\right)<\frac{t_i\ell_i}{i^2},
\end{align}
    \begin{align}\label{eq flat detail}
        t_{i-1}\ell_{i-1}\left(\lambda(\eta_{i-1}^*)+\frac{3}{i-1}\right)<\frac{t_i\ell_i}{i},\ 
        \frac{({i-1})^2i\ell_i}{\ell_{i-1}}\left(\lambda(\eta_i^*)+\frac{2}{i^2}\right)<t_{i-1},
    \end{align}
    \begin{align}\label{eq flat subset}
        \sum_{j=1}^{i-1}\frac{3 t_j\ell_j}{j^{1/2}}\leq \frac{t_i\ell_i}{i^{1/2}},\ 
    (|a|+\ell_i)|\alpha|+\sup_{0\leq p\leq|a|+\ell_{i}}\sup_{[F^{\ell_i}]}|S_p\phi^*|\leq t_{i-1}.
    \end{align}
For $i\in \mathbb{N}$ we set 
$
\hat{F}^{\ell_i}:=F^{\ell_i}a:=\{\tau a: \tau\in F^{\ell_i}\}.
$
For $i\in \mathbb{N}$ and $0\leq s\leq t_{i}$ we define 
\[
F^{\ell_i,0}=\{\emptyset\} \text{ and }F^{\ell_i,s}:=\prod_{k=1}^{s}\hat{F}^{\ell_i}=\{\tau_1\tau_2\cdots\tau_s:\tau_k\in \hat{F}^{\ell_i}, 1\leq k\leq s\}
\]
where $\emptyset$ denotes the empty word. For $i\in\mathbb{N}$ and $0\leq s\leq t_{i+1}$ we set
\begin{align*}
&E^{t_1,\cdots, t_i,s}:=\{\tau_1\cdots\tau_{i+1}:\tau_k\in F^{\ell_k,t_k}\text{ for all } 1\leq k\leq i\text{ and } \tau_{i+1}\in F^{\ell_{i+1},s}\}
\\&
\subset \MS^{\sum_{k=1}^it_k(\ell_k+|a|)+s(\ell_{i+1}+|a|)}.
\end{align*}
For $i\in \mathbb{N}$ we define the probability measure $\hat{\mu}_i^*$ on $\hat{F}^{\ell_i}$ by setting 
\begin{align}\label{eq def hat mu i}
\hat{\mu}_i^*(\{\tau\}):=(\#\hat{F}^{\ell_i})^{-1}=(\#F^{\ell_i})^{-1} \text{ for all }\tau\in \hat{F}^{\ell_i}.
\end{align}
Using Kolmogorov extension theorem, we define the product measure 
\[
\hat{\mu}^*:=\bigotimes_{k = 1}^{t_1}\hat{\mu}_1^*\times \bigotimes_{k = 1}^{t_2}\hat{\mu}_2^*\times \cdots \text{ on } V:=\prod_{k=1}^{t_1} \hat{F}^{\ell_1}\times\prod_{k=1}^{t_2} \hat{F}^{\ell_2}\cdots.
\]
We define the bijection map $\iota: V\rightarrow  \iota(V)$ by 
\[
\iota(\tau_1,\tau_2, \cdots):=\tau_1^1a\tau_2^1a\cdots\tau_{t_1}^1a
\tau_1^2a\tau_2^2a\cdots\tau_{t_2}^2a \cdots,
\]
where $\tau_k=\tau_1^ka\tau_2^ka\cdots\tau_{t_k}^ka$, $\tau_{j}^{k}\in F^{\ell_k}$, $1\leq j\leq t_k$ and $k\in \mathbb{N}$. We set 
\[
\mu:=\hat\mu^*\circ \iota^{-1}\circ\coding^{-1}.
\]
Note that for all $i\in\mathbb{N}$, $0\leq s\leq t_{i+1}$ and $\tau\in E^{t_1,\cdots,t_i,s}$ we have $\coding^{-1}(\coding([\tau]))\cap\iota(V)=[\tau]\cap\iota(V)$. 
The proof of this claim proceeds as follows. Let $\tau'\in \coding^{-1}(\coding([\tau]))\cap\iota(V)$. If $\tau'\notin[\tau]$ then there exists $1\leq k\leq |\tau|-1$ such that $\tau_k'\neq \tau_k$. Then by Lemma \ref{lemma boundary}, we have $\tau_{|\tau|-1}'\neq \tau_{|\tau|-1}=a_{|a|-1}$. On the other hand, since $\tau'\in \iota(V)$, we have $\tau'_{|\tau|-1}=a_{|a|-1}$. This is a contradiction. Thus $\tau'\in [\tau]\cap\iota(V)$. Hence, by \eqref{eq def hat mu i}, for all $i\in\mathbb{N}$, $0\leq s\leq t_{i+1}$ and $\tau=\tau_1\cdots\tau_{i+1}\in E^{t_1,\cdots,t_i,s}$ we obtain
\begin{align}\label{eq basic measure mu}
    \mu(\coding([\tau]))=(\#F^{\ell_1})^{-t_1}(\#F^{\ell_2})^{-t_2}\cdots(\#F^{\ell_i})^{-t_i}(\#F^{\ell_{i+1}})^{-s}.
\end{align}
Since $\{t_k\}_{k\in\mathbb{N}}$ is strictly increasing, 
for each integer $T\geq t_1$ there exists a unique pair of integers $N:=N_T$, $s:=s_T$ such that $T=\sum_{k=1}^{N}t_k+s$. We define 
\[
\Gamma:=\bigcap_{T=t_1}^\infty \bigcup_{\tau\in E^{t_1,\cdots,t_N,s}} A_{\tau}.
\]
We will show that 
\begin{align}\label{eq dim of Gamma}
    \dim_H(\Gamma)=\dimension.
\end{align}
Note that by \eqref{eq basic measure mu}, we have $\mu(\Gamma)>0$. Thus, by Frostman lemma (see, for example, \cite[Theorem 15.6.3]{Urbanskinoninvertible}), it is enough to show that for all $x\in \Gamma$ we have 
\begin{align}\label{eq Frostman}
\liminf_{r\to 0}\frac{\log \mu(B(x,r)\cap \Gamma)}{\log r}\geq \dimension.
\end{align}
For each integer $T\geq t_1$ we define 
\[
r_T:=\exp\left(-t_{N}\ell_{N}\left(\lambda(\eta_{N}^*)+3 N^{-2}\right)-s\ell_{N+1}\left(\lambda(\eta_{N+1}^*)+2(N+1)^{-2}\right)\right).
\]
If $s<t_{N+1}-1$ then $N_{T+1}=N_T$ and $s_{T+1}=s_T+1$, and so
\[
r_{T+1}=\exp(-\ell_{N+1}\left(\lambda(\eta_{N+1}^*)+2(N+1)^{-2}\right))r_T.
\]
If $s=t_{N+1}-1$ then $N_{T+1}=N_T+1$ and $s_{T+1}=0$, and so
\begin{align}\label{eq r_T exceptional}
r_{T+1}=\exp\left(-t_{N+1}\ell_{N+1}\left(\lambda(\eta_{N+1}^*)+3 (N+1)^{-2}\right)\right).
\end{align}
Note that for all $i\geq 2$ we have 
    \begin{align*}
        &t_{i}\ell_{i}\left(\lambda(\eta_{i}^*)+\frac{3}{i^2}\right)>t_{i-1}\ell_{i-1}\left(\lambda(\eta_{i-1}^*)+\frac{3}{(i-1)^2}\right)+(t_{i}-1)\ell_{i}\left(\lambda(\eta_{i}^*)+\frac{2}{i^2}\right)
    \\&\Longleftrightarrow \ell_{i}\left(\lambda(\eta_{i}^*)+\frac{t_i+2}{i^2}\right)>t_{i-1}\ell_{i-1}\left(\lambda(\eta_{i-1}^*)+\frac{3}{(i-1)^2}\right).
    \end{align*}
Hence, by \eqref{eq flat to define r_t}, the sequence $\{r_T\}_{T\geq t_1}$ is strictly decreasing to $0$ as $T\to \infty$. 
For $r\in (0,r_{t_1}]$ we take a unique integer $T:=T_r=\sum_{k=1}^{N_r}t_k+s_r\geq t_1$ such that $0\leq s_r< t_{N_r+1}$ and 
\begin{align}\label{eq r_T}
    r_{T+1}<r\leq r_T.
\end{align}
We first show that there exists a constant $C>0$ such that for all $r\in (0,r_{t_1}]$ and $x\in \Gamma$ we have 
$
    \#\{\tau\in E^{t_1,\cdots,t_N,s}:B(x,r)\cap \coding([\tau])\neq \emptyset\}\leq C.
$
 By \eqref{eq non intresting ine} and \eqref{eq nice cylinder appro}, for all $r\in (0,r_{t_1}]$ and $\tau\in E^{t_1,\cdots, t_N,s}$ we have 
\begin{align*}
    \sup_{\coding[\tau]}\log |(\rational^{|\tau|})'|\leq \sum_{k=1}^{N}t_k\ell_k( \lambda(\eta_k^*)+2k^{-2})+s\ell_{N+1}( \lambda(\eta_{N+1}^*)+2(N+1)^{-2}).
\end{align*}
Hence, by \eqref{eq to bound derivative} and the definition of $r_T$,
 for all $r\in (0,r_{t_1}]$ and $\tau\in E^{t_1,\cdots, t_N,s}$ we have 
\begin{align}\label{eq count ball 1}
    \sup_{\coding([\tau])}\log |(\rational^{|\tau|})'|\leq \log r_T^{-1}
    \Longrightarrow r\leq r_T\leq \inf _{\coding([\tau])}|(\rational^{|\tau|})'|^{-1}
    .
\end{align}
By the bounded distortion property (Lemma \ref{lemma distortion}), there exists $C_1>0$ such that for all $r\in (0,r_{t_1}]$, \[
 \tau=\tau_1\cdots\tau_N\tau_{N+1}^{1}a\cdots\tau_{N+1}^{s}a\in E^{t_1,\cdots,t_N,s}
 \]
 ($\tau_k\in F^{\ell_k,t_k}$ for $1\leq k\leq N$ and $\tau_{N+1}^{j}\in F^{\ell_{N+1}}\subset \MS^{\ell_{N+1}}(h)$ for $1\leq j\leq s$) and $x\in \coding([\tau_1\cdots\tau_N\tau_{N+1}^{1}a\cdots\tau_{N+1}^{s}])$ we have 
\begin{align}\label{eq count ball 2}
    B(x,C_1|(\rational^{|\tau|})'(x)|^{-1})\subset \comp(\tau_1\cdots\tau_N\tau_{N+1}^{1}a\cdots\tau_{N+1}^{s})
\end{align}
(see, for example, the proof of \cite[Lemma 19.3.10]{Urbanskinoninvertible}). Note that if $\tau'h,\tau h\in\MS^*$, $|\tau|=|\tau'|$ and $\tau\neq \tau'$ then $\comp(\tau h)\cap \comp(\tau'h)=\emptyset$. Hence, by \eqref{eq count ball 1} and \eqref{eq count ball 2},  there exists a constant $C>0$ such that for all $r\in (0,r_{t_1}]$ and $x\in \Gamma$ we have 
\[
    \#\{\tau\in E^{t_1,\cdots,t_N,s}:B(x,r)\cap \coding([\tau])\neq \emptyset\}\leq C.
\]
Combining this with \eqref{eq basic measure mu} and \eqref{eq nice cylinder appro}, for all $r\in (0,r_{t_1}]$ and $x\in \Gamma$ we obtain
\begin{align}\label{eq prob ball}
    &\mu(B(x,r)\cap\Gamma)\leq C(\#F^{\ell_N})^{-t_N}(\#F^{\ell_{N+1}})^{-s}
    \\&\leq C\exp\left(-t_N\ell_N(h(\eta_N^{*})-N^{-2})-s\ell_{N+1}(h(\eta_{N+1}^{*})-(N+1)^{-2})\right) \nonumber.
\end{align}
Thus, for all $r\in (0,r_{t_1}]$ with $s_{T_r}<t_{N_{T_r}+1}-1$ and $x\in \Gamma$, as $r\to 0$, 
\begin{align*}
    &\frac{\log\mu(B(x,r)\cap\Gamma)}{\log r}\geq \frac{\log (\#F^{\ell_N})^{-t_N}(\#F^{\ell_{N+1}})^{-s}}{\log r_{T+1}} +O\left((-\log r)^{-1}\right) 
    \\&\geq \frac{t_N\ell_N(h(\eta_N^*)-N^{-2})+s\ell_{N+1}(h(\eta_{N+1}^*)-(N+1)^{-2})}{t_{N}\ell_{N}(\lambda(\eta_{N}^*)+3 N^{-2})+(s+1)\ell_{N+1}(\lambda(\eta_{N+1}^*)+2(N+1)^{-2})}+O\left(({-\log r})^{-1}\right) \nonumber.
\end{align*}
Note that by \eqref{eq hyperbolic ergodic appro nice appro}, we have $h(\eta_{k}^*)-k^{-1}>0$ for all $k\in\mathbb{N}$. Hence, by \eqref{eq flat detail}, for all $r\in (0,r_{t_1}]$ with $s_{T_r}<N$ and $x\in \Gamma$, as $r\to 0$, 
\begin{align}\label{eq less than N}
    &\frac{\log\mu(B(x,r)\cap\Gamma)}{\log r}\geq \frac{h(\eta_N^*)-N^{-2}}{\lambda(\eta_{N}^*)+4 N^{-2}}+O\left((-\log r)^{-1}\right).
\end{align}
Moreover, for all $r\in (0,r_{t_1}]$ with $N\leq s_{T_r}<t_{N_{T_r}+1}-1$ and $x\in \Gamma$, as $r\to 0$,
\begin{align}\label{eq lager than N}
    \frac{\log\mu(B(x,r)\cap\Gamma)}{\log r}\geq&\min\left\{\frac{h(\eta_N^*)-N^{-2}}{\lambda(\eta_N^*)+3N^{-2}},\frac{N}{N+1}\frac{h(\eta_{N+1}^*)-(N+1)^{-2}}{\lambda(\eta_{N+1}^*)+2(N+1)^{-2}}\right\}\nonumber
    \\&+O\left((-\log r)^{-1}\right)
\end{align}
By \eqref{eq r_T exceptional} and \eqref{eq prob ball}, for all $r\in (0,r_{t_1}]$ with $s_{T_r}=t_{N_{T_r}+1}-1$ and $x\in \Gamma$, as $r\to 0$,
\[
\frac{\log\mu(B(x,r)\cap\Gamma)}{\log r}\geq\left(1-\frac{1}{t_{N+1}}\right)\frac{h(\eta_{N+1}^*)-(N+1)^{-2}}{\lambda(\eta_{N+1}^*)+3(N+1)^{-2}}+O((-\log r)^{-1}).
\]
Thus, by \eqref{eq less than N}, \eqref{eq lager than N}, \eqref{eq hyperbolic ergodic appro large dimension} and \eqref{eq hyperbolic ergodic appro nice appro}, for all $x\in \Gamma$ we obtain \eqref{eq Frostman}.
This completes the proof of \eqref{eq dim of Gamma}.

Next, we will show that 
\begin{align}\label{eq Gamma subset}
    \Gamma\subset B(\alpha).
\end{align}
We set $t_0:=0$ and $\ell_0:=0$. For $i\in\mathbb{N}$ and $0\leq m\leq t_i-1$ we define $q_{i,m}:=\sum_{k=1}^{i}t_{k-1}(\ell_{k-1}+|a|)+m(\ell_i+|a|)$.
By \eqref{eq hyperbolic ergodic appro nice appro}, \eqref{eq non intresting ine} and \eqref{eq nice cylinder appro}, for all $N\in\mathbb{N}$, $0\leq s\leq t_{N+1}-1$, $\tau\in E^{t_1,\cdots, t_N,s}$ and $x\in \coding([\tau])$ we have 
\begin{align}\label{eq essential Birkhoff}
   &| S_{|\tau|}(\phi)(x)-|\tau|\alpha|
   \\&
   \leq
   \sum_{i=1}^{N}\sum_{m=0}^{t_{i}-1}
   \left(
   \left|
   S_{\ell_i}(\phi)(\rational^{q_{i,m}}(x))
   -
   \ell_i\alpha\right|
   +
   \left|
   S_{|a|}(\phi)(\rational^{q_{i,m}+\ell_i}(x))
   -|a|\alpha\right|
   \right)\nonumber
   +\\&
   \sum_{m=0}^{s-1}\left(
   \left|
   S_{\ell_{N+1}}(\phi)(\rational^{q_{N+1,m}}(x))
   -
   \ell_{N+1}\alpha\right|
   +
   \left|
   S_{|a|}(\phi)(\rational^{q_{N+1,m}+\ell_{N+1}}(x))
   -|a|\alpha\right|
   \right)\nonumber
   \\&
   \leq \sum_{i=1}^N\frac{2t_i\ell_i}{i^{1/2}}
   +\left(\sum_{i=1}^{N}t_i+s\right)(c+|a||\alpha|)+
   \frac{2s\ell_{N+1}}{(N+1)^{1/2}}\leq
   \sum_{i=1}^N\frac{3t_i\ell_i}{i^{1/2}}+\frac{3s\ell_{N+1}}{(N+1)^{1/2}}.\nonumber
\end{align}
Let $x\in \Gamma$ and let $n\geq (\ell_1+|a|)t_1$. Then there exist $N:=N_n\in\mathbb{N}$, $0\leq s\leq t_{N+1}-1$ and $\tau\in E^{t_1,\cdots,t_N,s}$ such that $|\tau|\leq n<|\tau|+\ell_{N+1}+|a|$ and $x\in \coding([\tau])$. Then by \eqref{eq essential Birkhoff} and \eqref{eq flat subset}, we obtain
\begin{align*}
    &|S_n(\phi)(x)-n\alpha|\leq |S_n(\phi)(x)-S_{|\tau|}(\phi)(x)|+|S_{|\tau|}(\phi)(x)-|\tau|\alpha|+|\alpha||\tau|-n|
    \\&\leq|S_{n-|\tau|}(\phi)(\rational^{|\tau|}(x))|+ \sum_{i=1}^N\frac{3t_i\ell_i}{i^{1/2}}+\frac{3s\ell_{N+1}}{(N+1)^{1/2}}+|\alpha|(\ell_{N+1}+|a|)
\\&\leq\frac{4 t_N\ell_N}{N^{1/2}}+ \frac{3s\ell_{N+1}}{(N+1)^{1/2}}+t_{N}\leq \frac{4(t_N\ell_N+s\ell_{N+1})}{N^{1/2}}+\frac{t_N\ell_N}{\ell_N}\leq \frac{5n}{N^{1/2}}.
\end{align*}
Since $\lim_{n\to\infty} N_n=\infty$, we obtain $x\in B(\alpha)$. This completes the proof of \eqref{eq Gamma subset}. Combining \eqref{eq dim of Gamma} and \eqref{eq Gamma subset}, we obtain desired result.
\qed

\section{Verification of the Defining Condition of $\mathcal H$}\label{sec Verification of the Defining Condition}
As in the previous section, we fix  a parabolic rational map $\rational$ with $\text{deg}(\rational)\geq 2$ such that $\infty\notin J$ and \eqref{eq assumption simple} holds.
In this section, we give a sufficient condition for a continuous potential $\phi:J\to\mathbb{R}$ to belong to $\mathcal{H}$.

We say that for $\exponent>0$ and $A\subset J$ a potential $\phi:J\rightarrow \mathbb{R}$ is H\"older continuous on $A$ with exponent $\exponent$ if there exists a constant $C>0$ such that for all $x,y\in A$ we have 
$
|\phi(x)-\phi(y)|\leq C|x-y|^\exponent.    
$
A potential $\phi:J\to\mathbb{R}$ is said to be H\"older continuous if there exists $\exponent>0$ such that $\phi$ is H\"older continuous on $J$ with exponent $\exponent$. The number $\exponent$ is called a H\"older exponent of $\phi$. Recall that for $\omega\in \Omega$ we define
$
\growth=(\rate_\omega+1)\rate_\omega^{-1} \text{ and }
\minrate=\min_{\omega\in \Omega}\growth
$ (see \eqref{eq def minrate}).
\begin{lemma}\label{lemma sufficient condition Holder}
We assume that a potential $\phi:J\rightarrow \mathbb{R}$ is H\"older continuous and there exist $\epsilon>0$ and $\exponent>\minrate^{-1}$ such that $\phi$ is H\"older continuous on $B(\Omega,\epsilon)\cap J$ with exponent $\exponent$. Then $\phi\in \mathcal{H}$. 
\end{lemma}

\begin{proof}
Let $\exponent'$ be a H\"older exponent of $\phi$. Then there exists a constant $C_1>0$ such that for all $x,y\in J$ we have $|\phi(x)-\phi(y)|\leq C_1 |x-y|^{\exponent'}$ and for all $x,y\in B(\Omega,\epsilon)\cap J$ we have $|\phi(x)-\phi(y)|\leq C_1 |x-y|^{\exponent}$.
We denote by $C_2\geq 1$ the constant in Lemma \ref{lemma distortion}. 
By \eqref{eq n big near inddiferent}, there exists $N\in\mathbb{N}$ such that for all $n\geq N$ and $\parabolic\in \inddiferent^n\cap \MS^n$ we have $\coding([\parabolic])\subset B(\Omega,\epsilon)$. 
    By Lemma \ref{lemma exponential decay}, it is enough to show that there exist a constant $C>0$ and a exponent $\vartheta>0$ such that for all $e\in \induedge$ and $\tau,\tau'\in[e]$ we have 
    \begin{align}\label{eq varification}
        |\indupote(\tau)-\indupote(\tau')|\leq C|\coding\circ\Proj(\indushift(\tau))-\coding\circ\Proj(\indushift(\tau'))|^{\vartheta}.
    \end{align}

Let $n\in\mathbb{N}$ and let $e\in \induedge_n$. 
We fix $\tau,\tau'\in [e]$ and put $x:=\coding\circ\Proj(\indushift(\tau))$ and $y:=\coding\circ\Proj(\indushift(\tau'))$. 
We set $e^*:=\Proj^*(e)$.
Then we have 
$
\coding(\Proj(\tau))=\rational^{-n}_{e^*}(x)
\text{ and }
\coding(\Proj(\tau'))=\rational^{-n}_{e^*}(y).
$

If $n<N+1$ then by the mean value theorem and Lemma \ref{lemma distortion}, we have 
\begin{align}\label{eq varification 1}
    &\left|\indupote(\tau)-\indupote(\tau')\right|
    \leq C_1 \sum_{k=0}^{n-1} 
    |\rational^{-(n-k)}_{\shift^k(e^*)}(x)
    -
    \rational^{-(n-k)}_{\shift^k(e^*)}(y)
    |^{\exponent'}
    \leq C_1C_2^NDN|x-y|^{\exponent'},
\end{align}
where $D:=\sup_{1\leq n\leq N}\sup_{a\in \edge^n\times H\cap \MS^{n+1}}\sup_{z\in A_{a_{n}}}|(\rational^{-n}_{a})'(z)|^{\beta'}$.

We assume that $n\geq N+1$. By Proposition \ref{prop local property near inddiferent}, there exists a constant $C_3>0$ such that for all $\omega\in \Omega$ and $n\in\mathbb{N}$ we have  
\begin{align}\label{eq varification growth}
    D_{n,\omega}:=\sup_{a\in\edge\times \inddiferent^n_\omega\times H\cap \MS^{n+2}}\sup_{z\in A_{a_{n+1}}}|(\rational^{-n}_a)'(z)|\leq C_3
n^{-\growth}.
\end{align}
In particular, we have $D_{\sup}:=\max_{\omega\in \Omega}\sup_{n\in\mathbb{N}} D_{n,\omega}<\infty$.
By the mean value theorem, \eqref{eq varification growth}
and Lemma \ref{lemma distortion},
we have
\begin{align}\label{eq varification 2}
    &\left|\indupote(\tau)-\indupote(\tau')\right|
    \leq
    |\phi(\rational^{-n}_{e^*}(x))-\phi(\rational^{-n}_{e^*}(y))|
    \\&+
    \sum_{k=N}^{n-1}|\phi(\rational^{-k}_{\shift^{n-k}(e^*)}(x))-\phi(\rational^{-k}_{\shift^{n-k}(e^*)}(y))|\nonumber
    \\&+\sum_{k=1}^{N-1}|\phi(\rational^{-k}_{\shift^{n-k}(e^*)}(x))-\phi(\rational^{-k}_{\shift^{n-k}(e^*)}(y))|\nonumber
    \\&\leq
    C_1C_2^N\bigg(D_{\sup}|x-y|^{\exponent'}+C_3\sum_{k=1}^{\infty}k^{-\minrate\exponent}|x-y|^\exponent+DN|x-y|^{\exponent'}\bigg).\nonumber
\end{align}
Since $\minrate \exponent>1$, we have $\sum_{k=1}^{\infty}k^{-\minrate\exponent}<\infty$. Hence, by \eqref{eq varification 1} and \eqref{eq varification 2}, we obtain \eqref{eq varification}. Therefore, we are done. 
\end{proof}


\section*{Statements and Declarations}

\textbf{Funding}
The author was supported by the JSPS KAKENHI 25KJ1382.

\textbf{Competing Interests}
The author declares that there are no competing interests.

\textbf{Data Availability}
No datasets were generated or analysed during the current study.

\bibliographystyle{abbrv}
\bibliography{reference}

\end{document}